\documentclass{article}

\usepackage[english]{babel}
\usepackage{graphicx} 
\usepackage{ctex}
\usepackage{amsmath}
\usepackage{amsthm,amsmath,amssymb}
\usepackage{mathrsfs}
\usepackage{color}
\usepackage[letterpaper,top=2cm,bottom=2cm,left=3cm,right=3cm,marginparwidth=1.75cm]{geometry}
\newtheorem{theorem}{Theorem}[section]
\newtheorem{lemma}[theorem]{Lemma}
\usepackage{amsmath}
\usepackage{graphicx}
\usepackage[colorlinks=true, allcolors=blue]{hyperref}
\usepackage[numbers]{natbib}

\title{Oblique boundary value problems for $(n−1)$-type augmented Hessian equations}
\author{
  Zhibo Hu\textsuperscript{1}, Feida Jiang\textsuperscript{2,3} \\
  \small \textsuperscript{1}Center for Mathematics and Interdisciplinary Sciences, Fudan University, Shanghai 200433, PR China \\
   \small \textsuperscript{2}Shanghai Institute for Mathematics and Interdisciplinary Sciences (SIMIS), Shanghai 
200433, PR China \\
  \small \textsuperscript{3}School of Mathematics and Shing-Tung Yau Center of Southeast University, Southeast University, \\
  \small Nanjing 211189, PR China
}
\date{}

\begin{document}
\nocite{*}
\maketitle
\begin{abstract}
    In this paper, we study the global regularity of oblique boundary value problems for $(n−1)$-type augmented Hessian equations on a bounded domain, without imposing any convexity condition either on the domain or on the matrix-valued function in the $(n−1)$-type augmented Hessian equations. We establish a global existence and uniqueness theory for classical elliptic solutions by deriving global a priori estimates up to second-order derivatives. Besides the known applications for $(n-1)$ Monge-Ampère operators in optimal transportation and geometric optics, the general theory here embraces equations raised by Li-Sheng [Some Dirichlet problems arising from conformal geometry. Pac. J. Math. 251(2011), 337-359] and Guan-Qiu-Yuan [Fully nonlinear elliptic equations for conformal deformations of Chern--Ricci forms. Adv. Math. 343(2019), 538-566]. 
\end{abstract}

\textbf{Keywords:} Oblique boundary value problems$\,\cdot\,$$(n−1)$-type Augmented Hessian

equations$\,\cdot\,$Global second-derivative estimates$\,\cdot\,$ Gradient estimates$\,\cdot\,$Classical solvability.

\section{Introduction}
In this article, we consider oblique boundary value problems for  $(n-1)$-type augmented Hessian equations
\begin{equation}
\tag{1.1}
      \left\{
\begin{aligned}
&\mathcal{F}[u]:=F[\gamma\Delta uI-D^2u-A(x,u,Du)]=B(x,u,Du)\ \ \mathrm{in} \  \Omega,\\
    &\mathcal{G}[u]:=G(x,u,Du)=0\ \ \mathrm{on} \ \partial\Omega,     
\end{aligned}
\right.
\end{equation}
whose equation was raised by Li-Sheng \cite{li2011some} in 2011, and Guan-Qiu-Yuan \cite{guan2019fully} dealt with an equivalent form of this equation in 2019. They solved Dirichlet problems about this equation. In conformal geometry, many problems can be reduced to an oblique boundary value problem for a nonlinear equation \cite{Chen_2009,sheng2013k}. Therefore, it is natural to consider the existence results for the equations in (1.1) under oblique boundary conditions.

In (1.1), $\Omega$ is a bounded domain in $\mathbb{R}^n$ with smooth boundary, $\gamma>1$ which is consistent with the original paper. $Du$ and $D^2u$ denote the gradient vector and the Hessian matrix of solution $u\in C^2(\Omega)$. A is a $n\times n$ symmetric matrix function defined on $\Omega\times \mathbb{R}\times \mathbb{R}^n$, B is a scalar valued function on $\Omega\times \mathbb{R}\times \mathbb{R}^n$. Let $x,z,p,r$ represent the points in $\Omega,\mathbb{R},\mathbb{R}^n$ and $\mathbb{R}^{n\times n}$ respectively. 

The meaning of oblique boundary condition for $u\in C^1({\overline{\Omega}})$ is when $G\in C^1(\partial\Omega\times \mathbb{R}\times \mathbb{R}^n)$, it holds
\begin{equation}
\tag{1.2}
    G_p(x,u,Du)\cdot\nu(x)\ge\beta_0, \ \ \mathrm{on} \ \partial\Omega,
\end{equation}
where $\nu$ is the unit inner normal vector field on $\partial\Omega$ and $\beta_0$ is a positive constant. If $G_p\cdot \nu(x)>0$ on all of $\partial\Omega\times \mathbb{R}\times \mathbb{R}^n$, we will simply refer to $G$ as oblique.

Let $\mathbb{S}^n$ represent the linear space of $n\times n$ symmetric matrices, function $F$ in $(1.1)$ is defined on an open, convex cone $\Gamma$ in $\mathbb{S}^n$, with vertex at 0, containing the positive cone $\Gamma^+$. As in Li-Sheng \cite{li2011some}, we can let $\Gamma^+$ satisfy $\Gamma_n^+\subset\Gamma^+\subset\Gamma\subset\Gamma_1^+$. For convenience, we assume that $F\in C^2(\Gamma)$ satisfies a subset of the following properties.

\vspace{0.5cm}
\noindent \textbf{F1} \ \ $F$ is monotone increasing in $\Gamma$, that is
\begin{equation}
\tag{1.3}
    F_r:=F_{r_{ij}}=\{F^{ij}\}=\{\frac{\partial F}{\partial r_{ij}}\}>0, \, \mathrm{in} \  \Gamma.
\end{equation}

\noindent \textbf{F2} \ \ $F$ is concave in $\Gamma$, that is
\begin{equation}
\tag{1.4}
    \frac{\partial^2F}{\partial r_{ij}\partial r_{kl}}\eta_{ij}\eta_{kl}\le0,\, \mathrm{in} \  \Gamma,
\end{equation}
\indent for all symmetric matrices $\{\eta_{ij}\}\in \mathbb{S}^n$.

\noindent \textbf{F3}\ \ $F(\Gamma)=(a_0,\infty)$ for constant $a_0\ge-\infty$ with \begin{equation}
    \tag{1.5}
    \underset{r_0\in\partial\Gamma}{\mathrm{min}} \underset{r\to r_0}{\mathrm{limsup}}F(r)=a_0.
\end{equation}

\noindent \textbf{F4}\ \ for given constants $a,b$ satisfying $a_0<a<b$, there exists a constant $\delta_0>0$ such that

\indent $\mathcal{F}(r):=\mathrm{trace}(F_r)\ge\delta_0$ if $a<F(r)<b$.

\noindent \textbf{F4}$^+$\ \ $\mathcal{F}(r)\to \infty$ uniformly for $a\le F(r)\le b$ as $|r|\to\infty$.

If $F$ satisfies the above properties, we say the operator $\mathcal{F}$ also satisfies them. Note that when $\Gamma$ is closed under orthogonal transformations and $F$ is a symmetric function of the eigenvalues $\lambda_1,...\lambda_n$ of the matrix $r$, we say $\mathcal{F}$ is orthogonally invariant. When $A=0$, the above conditions are modelled on the conditions introduced for the study of the Dirichlet problem for the basic Hessian equations, by Caffarelli et al. \cite{caffarelli1985dirichlet} and Trudinger \cite{trudinger1995dirichlet}.

So we also assume that the operator $F$ is homogeneous of degree one and for any $r\in\Gamma$, then it can be shown that
\begin{equation}
    \tag{1.6}
    \sum_{i,j}r_{ij}F^{ij}=F.
\end{equation}
For the present problem, correspondingly, the function $B$ on the right-hand side of the equation (1.1) must also satisfy a similar condition.

As is known, the $(n−1)$ Monge-Ampère equation which was introduced by Harvey-Lawson in a series of papers \cite{harvey2011dirichlet,harvey2012geometric}. Following them, we say a $(n−1)$ operator refers to the operator whose variable is $\Delta u-D^2u$. More information can be found in \cite{fu2015form,jiao2023class,chu2023fully,dong2026dirichlet} and references therein. The $(n-1)$-type operator in this paper is defined as an operator whose variable is
\begin{equation}
    \notag
    \gamma\Delta u-D^2u, \, \gamma>1.
\end{equation}
The term ``augmented'' refers to the presence of an additional matrix $A(x,z,p)$. Next we recall some results about oblique boundary value problems for fully nonlinear elliptic partial differential equations.

Following previous works that considered the case $\gamma>1$, an important observation is made: $F$ is a uniformly elliptic operator, in the sense of the definition in \cite{lieberman1986nonlinear}. This is straightforward to verify.

We set 
\begin{equation}
    \tag{1.7}
    \widetilde{F}^{ij}=\frac{\partial F}{\partial u_{ij}}.
\end{equation}
Then it has
\begin{equation}
    \tag{1.8}
   \widetilde{F}^{ij}=\delta_{ij}\gamma\sum_{k=1}^nF^{kk}-F^{ij}.
\end{equation}

So the eigenvalues of matrix $\{\widetilde{F}^{ij}\}_{n\times n}$ are $\gamma\mathcal{F}-\lambda_i$, where $\lambda_i,i=1,2...n$ are the eigenvalues of matrix $\{F^{ij}\}_{n\times n}$. We note $\lambda_{\min}=\min\{\lambda_1,\lambda_2,..\lambda_n\},\lambda_{\max}=\max\{\lambda_1,\lambda_2,..\lambda_n\}$. Then the maximum and minimum eigenvalues of $\{\widetilde{F}^{ij}\}_{n\times n}$ are $\gamma\mathcal{F}-\lambda_{\min}$ and $\gamma\mathcal{F}-\lambda_{\max}$. This means 
\begin{equation}
    \tag{1.9}
    \frac{\gamma\mathcal{F}-\lambda_{min}}{\gamma\mathcal{F}-\lambda_{max}}\le\frac{\gamma\mathcal{F}}{(\gamma-1)\mathcal{F}}=\frac{\gamma}{\gamma-1}.
\end{equation}
So the operator $F$ is uniformly elliptic.

The main novelty of this paper is the systematic use of the uniform inequality, We found that using either notation $\{F^{ij}\}$ or $\{\widetilde{F}^{ij}\}$ alone in the proof introduces considerable complexity into the computations, making it difficult to reach a concise conclusion. However, when the two notations $\{F^{ij}\}$ and $\{\widetilde{F}^{ij}\}$ are used in conjunction, each leveraging its respective advantages, the efficiency of the proof is greatly enhanced. In fact, the systematic use of the notation $\{\widetilde{F}^{ij}\}$ amounts to the systematic exploitation of the uniform ellipticity induced by $\gamma>1$.
\vspace{0.3cm}

The oblique boundary value problem has a research history spanning over four decades, first studied by Trudinger, Lieberman, and Lions in the 1980s. Among their works, Lions and Trudinger \cite{lions1986linear} investigated the oblique boundary problem for linear equations, while Lieberman and  Trudinger \cite{lieberman1986nonlinear} established a general framework for treating oblique boundary value problems for quasilinear and fully nonlinear uniformly elliptic equations, setting a benchmark for subsequent studies. Later, in the 1990s, Urbas \cite{urbas1995nonlinear,urbas1996nonlinear,urbas1998oblique} studied the oblique boundary value problem for Hessian, curvature, and Monge–Ampère equations, proving the existence of solutions under certain conditions on the domain and the boundary. Subsequently, Gilbarg and Trudinger included this class of problems in Chapters 6 and 17 of their monograph \cite{gilbarg1977elliptic}.

During the 2000s–2010s, the oblique boundary value problem attracted relatively little attention, until around 2015, when Jiang and Trudinger \cite{jiang2018oblique,jiang2017oblique,jiang2019oblique} studied the oblique boundary value problem for augmented Hessian equations. Under suitable convexity, regularity, and smoothness assumptions on the domain and the functions involved, they established the existence of classical solutions, thereby reopening research on oblique boundary value problems for fully nonlinear elliptic equations. In the period 2010–2020, Jiang, Trudinger and their collaborators \cite{jiang2020dirichlet,jiang2014dirichlet,jiang2015dirichlet,jiang2016neumann} produced a substantial body of work on augmented Hessian equations and related problems.

Following these developments, a growing number of researchers have investigated the oblique boundary value problem for fully nonlinear elliptic equations. In particular, Wang \cite{wang2022gradient} obtained global gradient estimates for Hessian equations with oblique boundary conditions. Dong-Wei \cite{Dong2022Neumann} investigated a special case of the oblique boundary value problem for $(n-1)$ augmented Hessian equations, namely when the function $A=0$ and the boundary condition satisfies $G_p=\nu$, i.e., the Neumann boundary condition. They obtained the existence and uniqueness of admissible solutions to the equations by establishing a priori estimates up to the second order derivatives. Xu and Xu \cite{xu2017gradient} obtained global gradient estimates for prescribed mean curvature equation under semi-linear oblique boundary value condition. Han \cite{han2025tug} studied the oblique boundary value problem for the normalized $\mathit{p}$-Laplacian equation and applied it in the context of stochastic theory. Bessa and Ricarte \cite{bessa2023global} developed weighted Lorentz–Sobolev estimates for viscosity solutions of fully nonlinear elliptic equations with oblique boundary conditions under weakened convexity assumptions, and applied this method to derive Lorentz–Sobolev estimates for solutions. Xiang-Xiong-Zheng \cite{xiang2026oblique} studied oblique boundary value problems for degenerate Hessian quotient type equations. By constructing suitable a priori estimates, they established the existence and uniqueness of solutions and provided $C^{1,1}$ estimates.

Before we state our theorems, we denote by $C$ a constant in this paper that may vary in different contexts. For simplicity, we denote all of them uniformly by $C$.

The results of the present paper can, to some extent, be regarded as a continuation of the series of studies on oblique boundary value problems by Jiang and Trudinger \cite{jiang2018oblique,jiang2017oblique,jiang2019oblique}, enriching the theoretical framework they developed.

Now we start to state the main theorems in this paper. First we state the local and global second derivative estimate which extends the result in \cite{jiang2018oblique}. The global second derivative estimate is necessary for dealing with the boundary condition (1.2).

\vspace{0.2cm}
\begin{theorem}
  Let $u\in C^4(\Omega)$ be an admissible solution of equation $\mathrm{(1.1)}$ in $\Omega$ and $B\in C^2(\overline{\Omega}\times \mathbb{R}\times \mathbb{R}^n)$ satisfies $B>a_0$ on the range of $(x,u,Du)$. Assume one of the following conditions holds:
\vspace{0.1cm}
\noindent (i) $F1,F2,F3,F4^+$ hold;

\noindent (ii) $F1,F2,F3,F4$ hold, and $B$ is convex with respect to $p$.

Then we have estimate 
\begin{equation}
    \tag{1.10}
    \underset{\Omega}{\sup}|D^2u|\le C+n\, \underset{\partial\Omega}{\sup}|D^2u|,
\end{equation}
for $u\in C^4(\overline{\Omega})$, where the constant $C$ depends on $n,A,B,\gamma,F,\Gamma,\Omega$ and $||u||_{C^1;\Omega}$.  
\end{theorem}

\vspace{0.4cm}

For the second derivative estimate on boundary, more assumptions are needed. As in \cite{jiang2018oblique}, we assume the function $G\in C^2(\partial\Omega\times \mathbb{R}\times \mathbb{R}^n)$ is concave in $p$, which means $G_{pp}\le0$ on $\partial\Omega\times \mathbb{R}\times \mathbb{R}^n$. It is easy to see that the quasilinear case is a special case of concavity, that is 
\begin{equation}
    \tag{1.11}
    G(x,z,p)=\beta(x,z)\cdot p-\varphi(x,z).
\end{equation}
At this moment, $\beta=G_p$ and $\varphi$ are defined on $\partial\Omega\times \mathbb{R}$. For a function $u\in C^1$, if $G_{pp}(x,u,Du)\le0$ on $\partial\Omega$, we say $G$ is concave in $p$ with respect to $u$. 

Finally, to complete our hypotheses, we assume that the cone $\Gamma$ lies strictly in a half space in the sense that $r\le \mathrm{trace}(r) I$ for all $r\in\Gamma$, that is $\Gamma\subset\mathcal{P}_{n-1}:=\{r\in\mathbb{S}^n:\mathrm{trace}(r) I-r\ge0\}$, which is in accordance with the classical Hessian equation problems whose cones are $\Gamma_k$ for $k\ge2$. Now we state the result of global second derivative estimation

\vspace{0.4cm}

\begin{theorem}
 Let $u\in C^4(\Omega)\cap C^3(\overline{\Omega})$ be an admissible solution of the boundary value problem (1.1) in a $C^3$ domain $\Omega\subset \mathbb{R}^n$. Assume that $F$ satisfies conditions $F1-F4$, $B\in C^2(\overline{\Omega}\times \mathbb{R}\times \mathbb{R}^n)$ satisfies $B>a_0$ on the range of $(x,u,Du)$., $G\in C^2(\partial\Omega\times \mathbb{R}\times \mathbb{R}^n)$ is oblique and quasilinear, $\Gamma\subset\mathcal{P}_{n-1}$ and either $F4^+$ holds or $B$ is independent of $p$. Then we have the estimate
\begin{equation}
    \tag{1.12}
    \underset{\Omega}{\sup}|D^2u|\le C,
\end{equation}
where $C$ is a constant depending on $F,\gamma,A,B,G,\Omega,\beta_0$ and $||u||_{1,\Omega}$.   
\end{theorem}

\vspace{0.4cm}
In order to make use of Theorem 1.2 to prove the existence of smooth solutions to $(1.1)$, gradient and solution estimates are also needed. Our gradient estimates are motivated by Jiang and Trudinger in \cite{jiang2017oblique,jiang2019oblique}. But using the condition of uniform ellipticity we can greatly simplify the calculation process and improve the conditions we need to assume.
 
First we introduce the growth conditions on $A$ and $B$ as follows.
\begin{equation}
    \tag{1.13}
    \begin{split}
         D_xA,D_xB,D_zB,D_zA=o(|p|^3),\\
         D_zA\ge o(|p|^2)I,D_zB\ge o(|p|^2),\\
         D_pA,D_pB=o(|p|^2),
    \end{split}
\end{equation}
as $|p|\to\infty$, uniformly for $x\in\Omega,|z|\le M$ for any $M>0$.
Namely, there exists a monotone function $\omega(t)\to 0$, such that 
\begin{equation}
    \notag
    \begin{split}
         |D_xA|+|D_xB|+|D_zB|+|D_zA|\le\omega(|p|)|p|^3,\\
         D_zA\ge \omega(|p|)|p|^2I,D_zB\ge \omega(|p|)|p|^2,\\
         |D_pA|+|D_pB|\le \omega(|p|)|p|^2.
    \end{split}
\end{equation}

as $|p|\to\infty$, uniformly for $x\in\Omega,|z|\le M$ for any $M>0$. Note that the analogous natural growth conditions appear in many articles and theories, such as in \cite{gilbarg1977elliptic}.

Now we give a gradient estimate for oblique semilinear boundary conditions, that is when $\beta$ in (1.11) is independent of $z$ so that the second equation (1.1) can be written as follows.
\begin{equation}
    \tag{1.14}
    \mathcal{G}[u]=G(x,u,Du)=\beta(x)\cdot Du-\varphi(x,u)=0, \,  \mathrm{on} \  \partial\Omega
\end{equation}

\vspace{0.4cm}
\begin{theorem}
Let $u\in C^3(\Omega)\cap C^2(\overline{\Omega})$ be an admissible solution of the boundary value problem $\mathrm{(1.1)}$ for an oblique, semilinear boundary operator $\mathcal{G}$ in a bounded $C^3$ domain $\Omega\subset \mathbb{R}^n$. Assume that $\mathcal{F}$ satisfies $F1$, $F3$ and $F4$ with $b=\infty$, $A,B\in C^1(\overline{\Omega}\times \mathbb{R}\times \mathbb{R}^n)$, satisfy $\mathrm{(1.13)}$, $b_0:=\underset{\Omega}{\inf}\,B(x,u,Du)>0$ and $\beta\in C^2(\partial\Omega), \varphi\in C^2(\partial\Omega\times \mathbb{R})$. Assume also $p\cdot A_p\le O(|p|^2)I, \ \ p\cdot B_p\le O(|p|^2), A=O(|p|^2)$ in $\mathrm{(1.13)}$. Then we have the estimate
\begin{equation}
    \tag{1.15}
    \sup_\Omega|Du|\le C,
\end{equation}
Where $C$ is a constant depending on $F,A,\gamma,B,\Omega,b_0,\beta,\varphi$ and $||u||_{{C^0};\Omega}$.

\end{theorem}
\vspace{0.4cm}

An additional feature is that we don't need the concavity condition $F2$ in Theorem 1.3. It can be removed when $F$ is homogeneous of degree one and more generally. We will illustrate it in Sect. 3.   

\vspace{0.4cm}
\noindent \textbf{Remark 1.1} By virtue of uniform ellipticity of $\mathcal{F}$, The hypotheses we make in Theorem 1.1-1.3 are much simpler when compared with the hypothesis in \cite{jiang2018oblique}. It replaces two structural conditions and  weakens the assumption $A=o(|p|^2)$ to $A=O(|p|^2)$  used in the general augmented Hessian theory. Two structural conditions are the strict regularity of $A$ in the gradient variables and uniform $(\Gamma,A,G)$-convex with respect to $u$ of the boundary. More details, we exploit it at three points: to absorb the $A_{pp}$ terms in the interior $C^2$ estimate, to construct distance function barriers without a sign condition on the boundary curvature, and to control
the Hessian terms in the global gradient estimate.

\vspace{0.4cm}

After we get the results Theorem 1.2 and Theorem 1.3, to complete the proof of the existence theorem, it remains to establish the necessary solution estimates. As in \cite{jiang2018oblique}, we will assume the existence of subsolutions and supersolutions to compensate for the lack of solution estimates. For this aim, we define that the functions $\underline{u}$ and $\overline{u}$, in $C^2(\Omega)\cap C^1(\Omega)$, are a subsolution and a supersolution of the boundary value problem $(1.1)$, if they are admissible and satisfy
\begin{equation}
    \tag{1.16}
     \left\{
    \begin{aligned}
&\mathcal{F}[\underline{u}]\ge B(x,\underline{u},D\underline{u}),\,  \mathrm{in} \ \Omega,\\
    &\mathcal{G}[\underline{u}]\ge0,\,  \mathrm{on} \ \partial\Omega.    
\end{aligned}
\right.
\end{equation}
and
\begin{equation}
    \tag{1.17}
    \left\{
    \begin{aligned}
&\mathcal{F}[\overline{u}]\le B(x,\overline{u},D\overline{u}),\ \ \mathrm{in}\ \ \Omega,\\
    &\mathcal{G}[\overline{u}]\le0,\ \ \mathrm{on} \ \partial\Omega.     
\end{aligned}
\right.
\end{equation}
at every point. 

With the above results, we state the existence result for the classical admissible solutions of the oblique boundary value problems for $(n−1)$-type augmented Hessian equations (1.1).
\vspace{0.4cm}
\begin{theorem}
Assume that $F$ satisfies conditions $F1-F4$ with $b=\infty$ in the cone $\Gamma\subset\mathcal{P}_{n-1}$, $\Omega$ is a bounded $C^{3,1}$ domain in $\mathbb{R}^n$, $A\in C^2(\overline{\Omega}\times \mathbb{R}\times \mathbb{R}^n)$, $B\in C^2(\overline{\Omega}\times \mathbb{R}\times \mathbb{R}^n)$ satisfies $B>0$, $\mathcal{G}$ is oblique and semilinear, with $\beta\in C^{2,1}(\partial\Omega),\varphi\in C^{2,1}(\partial\Omega)$, $\underline{u}$ and $\overline{u},\in C^2(\Omega)\cap C^1(\Omega)$ are an admissible subsoltion and a supersolution for the oblique boundary value problem $\mathrm{(1.1)}$. Assume also that $A,B$ and $\varphi$ are non-decreasing in $\mathrm{z}$, with at least one of them strictly increasing, and that $A$ and $B$ satisfy the growth conditions $\mathrm{(1.14)}$. Assume that either $F4^+$ holds or $B$ is independent of $p$ and $p\cdot A_p\le O(|p|^2)I, \ \ p\cdot B_p\le O(|p|^2), A=O(|p|^2)$ in $\mathrm{(1.13)}$. Then there exists a unique admissible soultion $u\in C^{3,\alpha}(\Omega)$ for the oblique boundary value problem $\mathrm{(1.1)}$, for any $\alpha\in(0,1)$.
\end{theorem}

The existence result can be proved by the method of continuity, and the uniqueness follows directly by applying our monotonicity conditions on $A,B$ and $G$ to the comparison principle. It is worth mentioning that the higher regularity assumptions about $\beta$ and $\varphi$ are necessary in the application of the continuity method.

\vspace{0.4cm}
\noindent \textbf{Remark 1.2} The problem (1.1) in this paper is a special variant of problem (1.1)-(1.2) in \cite{lieberman1986nonlinear}. Lieberman and Trudinger considered a more general class of uniformly elliptic equations and under fairly stringent assumptions, established the existence and uniqueness of the solution. The main difference is that our equation is nonhomogeneou, more precisely: the right-hand side is a function $B(x,u,Du)$ that may be highly complicated. Moreover we require fewer growth conditions than those in \cite{lieberman1986nonlinear}, indeed, the growth condition we impose is, in certain respects, weaker than theirs. The comparison is as follows.

For clarity, let $F'_1$--$F'_5$ denote the assumptions $F_1$--$F_5$ 
of \cite{lieberman1986nonlinear}. The comparison is as follows.

\begin{itemize}
  \item $F'_1$: the function $\mu(|z|)$ reduces to the constant 
        $\gamma/(\gamma-1)$ in our setting.
  \item $F'_5$: coincides with assumption $F_2$ in this paper.
  \item $F'_2$ and $F'_4$: not required here, these are growth assumptions 
        on the operator $F$ and its second-order derivatives.
  \item $F'_3$: the conditions $p\cdot A_p \le O(|p|^2)I$ and 
        $p\cdot B_p \le O(|p|^2)$ coincide with $F'_3$, whereas 
        $D_xA, D_zA = o(|p|^3)$ and $D_pA = o(|p|^2)$ are slightly stronger than 
        $F'_3$ when the term $|r|$ is absent from the right-hand side of $F'_3$. 
        In addition, we impose extra but admissible conditions such as $D_zA \ge o(|p|^2)I$, $A = O(|p|^2)$, and certain conditions on $B$ stated in (1.13).
\end{itemize}

\vspace{0.4cm}
\noindent \textbf{Remark 1.3} For this problem, the cases $\gamma>1$ and $\gamma\ge1$ are quite different. Clearly, $\gamma\ge1$ covers more situations. However, when $\gamma=1$,the equation is only elliptic and cannot be shown to be strictly elliptic, which greatly increases the difficulty of the proof. It is mentioned in \cite{xiang2026oblique}, for degenerate Hessian quotient type equations 
\begin{equation}
\notag
\frac{\sigma_k}{\sigma_l}(\gamma\Delta I-D^2u)=f(x,u),
\end{equation}
when $0\le l<k<n$, the existence and uniqueness of solutions for oblique boundary value problems can also be obtained in the case $\gamma=1$. But when $0\le l<k\le n$, the existence and uniqueness of solutions can only be obtained under the condition $\gamma>1$. In order not to introduce additional boundary convexity conditions and regularity conditions on the functions in the equation, and to remain consistent with the form of the equations in \cite{li2011some,guan2019fully}, we choose to study the case $\gamma>1$.
\vspace{0.4cm}

The paper is organised as follows. In Sect. 2 we collect preliminary notation and identities, which include Lemma 2.1 and some necessary calculations for the proofs. In Sect. 3, we first prove the interior second derivative estimate, Theorem 1.1. Then, through classifying the second-order derivatives into three categories, which are mixed tangential-oblique direction($D_{\tau\beta}u$), double oblique direction($D_{\beta\beta}u$) and double tangential direction($D_{\tau\tau}u$), we get the Lemmas 3.1 and 3.2. Naturally we get the estimate of the second derivative on boundary and prove the Theorem 1.2. In Sect. 4, we prove the global gradient estimate, Theorem 1.3, under some more general growth assumptions on $A$ and $B$. Finally, In Sect. 5, we prove the existence theorem, Theorem 1.4 for semilinear oblique boundary value problems based on the priori derivative estimates.

\section{Preliminaries}
In this section, we collect preliminary notation and identities for the proofs in the following sections. The proofs in Section 3 can be greatly simplified by virtue of these preliminary works.

First of all, we introduce some notation for the admissible solutions of (1.1). We denote the $(n-1)$ type augmented Hessian $M[u]$ by $w=w_{ij}$, they are

    \begin{equation}
    \tag{2.1}
        w_{ij}=\gamma\delta_{ij}\Delta u-u_{ij}-A_{ij},\ \ M[u]=\gamma\Delta uI-D^2u-A. 
    \end{equation}
Corresponding, the first and second partial derivatives of $F$ at $M[u]$ by $F^{ij}$ and $ F^{ij,kl}$ are denoted as
\begin{equation}
\tag{2.2}
    F^{ij}=\frac{\partial F}{\partial r_{ij}}(M[u]), \ \ F^{ij,kl}=\frac{\partial^2 F}{\partial r_{ij}\partial r_{kl}}(M[u]).
\end{equation}
    
According to the Assumptions $F1$ and $F2$, we know that the matrix $\{F^{ij}\}$ is positive definite and  
\begin{equation}
\notag
    F^{ij,kl}\eta_{ij}\eta_{kl\le0}
\end{equation}
holds for any admissible solution $u$ and all $\{\eta_{ij}\}\in \mathbb{S}^n$. And we also denote $\mathcal{F}=\mathrm{trace}(F_r)=\sum_i^nF^{ii}$ which is consistent with the above notation. 

Second, an important Lemma should be introduced 
\begin{lemma}
 Let $\vec{x}=(x_1,x_2,...x_n)\in \mathbb{R}^n$, $\gamma>1$, $\{F^{ij}\}_{n\times n}$ is a positive definite matrix, it holds that
\begin{equation}
    \tag{2.3}
    F^{ij}\gamma\delta_{ij}\sum_mx_m^2-F^{ij}x_ix_j\ge(\gamma-1)\mathcal{F}|\vec{x}|^2\ge(\gamma-1)F^{ij}x_ix_j\ge0.
\end{equation}   
\end{lemma}
\begin{proof}
    Since $\{F^{ij}\}_{n\times n}$ is a positive definite matrix, for any vector $\vec{x}=(x_1,x_2,...,x_n)$, it holds that 
    \begin{equation}
        \notag
        \vec{x}^T\{F^{ij}\}_{n\times n}\vec{x}=F^{ij}x_ix_j\le\lambda_{\max}|\vec{x}|^2\le\mathcal{F}|\vec{x}|^2.
    \end{equation}

So    
    \begin{equation}
        \notag
        \gamma F^{ii}|\vec{x}|^2-F^{ij}x_ix_j\ge(\gamma-1)\mathcal{F}|\vec{x}|^2\ge(\gamma-1)F^{ij}x_ix_j\ge0.
    \end{equation}
    
\end{proof}

Third, we define two linearized operators of the operator $\mathcal{F}$ and Eq.(1.1),
\begin{equation}
    \tag{2.4}
    \begin{split}
&Lv:=F^{ij}(\gamma\delta_{ij}\Delta v-v_{ij}-A_{ij}^{P_k}D_kv),\\
&\mathcal{L}v:=Lv-D_{P_k}B(x,u,Du)D_kv,
    \end{split}
    \end{equation}
for any functions $v\in C^2(\Omega)$, where $A_{ij}^{P_k}=D_{P_k}A(x,u,Du)$ and sometimes hereinafter referred to as $A_{ij}^k$. According to (1.9), the operators $L$ and $\mathcal{L}$ is uniformly elliptic.

Next, for convenience, we will as usual denote the partial derivatives of admissible functions $u$ on $\Omega$ as $u_i=D_{x_i}u,u_{ij}=D_{x_ix_j}u$. And for a vector $\tau$, we denote the directional derivatives along the vector $\tau$ are $u_\tau=D_\tau u=u_i\tau^i,u_{i\tau}=u_{ij}\tau^j,u_{\tau\tau}=u_{ij}\tau^i\tau^j$, etc. Differentiating equation (1.1) once for variable $x_m$, we obtain
\begin{equation}
    \tag{2.5}
   F^{ij}(\delta_{ij}\gamma\Delta u_m-u_{ijm}-\widetilde{D}_{x_m}A_{ij}-A_{ij}^{P_k}u_{km})=\widetilde{D}_{x_m}B+B^{P_k}u_{km},
\end{equation}
where the operator $\widetilde{D}_{x_m}:=D_{x_m}+u_mD_z$, So similarly we denote the operator $\widetilde{D}_{x}:=D_{x}+DuD_z$.

Then, Differentiating again for variable $x_n$, this yields
\begin{equation}
    \tag{2.6}
    \begin{split}
    &F^{ij,kl}(\delta_{ij}\gamma\Delta u_m-u_{ijm}-\widetilde{D}_{x_m}A_{ij}-A_{ij}^{P_k}u_{km})(\delta_{kl}\gamma\Delta u_n-u_{kln}-\widetilde{D}_{x_n}A_{kl}-A_{kl}^{P_o}u_{on})+\\
    &F^{ij}(\gamma\delta_{ij}\Delta u_{nm}-u_{ijnm}-\widetilde{D}_{x_n}\widetilde{D}_{x_m}A_{ij}-D_{P_l}\widetilde{D}_{x_m}A_{ij}u_{ln}-\widetilde{D}_{x_n}A_{ij}^ku_{km}-A_{ij}^{kl}u_{ln}u_{km}\\
&-A_{ij}^ku_{kmn})=\widetilde{D}_{x_n}\widetilde{D}_{x_m}B+D_{P_l}\widetilde{D}_{x_m}Bu_{ln}+\widetilde{D}_{x_n}B^{P_k}u_{km}+B^{P_kP_l}u_{ln}u_{km}+B^{P_k}u_{kmn},
    \end{split}
\end{equation}
where the $A_{ij}^{kl}(x,u,Du)$ means $D_{p_kp_l}A_{ij}$. 

Using the (2.5) and (2.6), we can simplify the formulas $\mathcal{L}u_m$ and $\mathcal{L}u_{nm}$. Immediately, we will get 
\begin{equation}
    \tag{2.7}
    \mathcal{L}u_m= F^{ij}(\gamma\delta_{ij}\Delta u_m-u_{ijm}-A_{ij}^{P_k}u_{km})-D_{P_k}B(x,u,Du)u_{km}=F^{ij}\widetilde{D}_{x_m}A_{ij}+\widetilde{D}_{x_m}B,
\end{equation}
and
\begin{equation}
    \tag{2.8}
    \begin{split}
          \mathcal{L}u_{nm}= &-F^{ij,kl}(w_{kl})_n(w_{ij})_m+F^{ij}(\widetilde{D}_{x_n}\widetilde{D}_{x_m}A_{ij}+D_{P_l}\widetilde{D}_{x_m}A_{ij}u_{ln}+\widetilde{D}_{x_n}A_{ij}^ku_{km}+\\
&A_{ij}^{kl}u_{ln}u_{km})+\widetilde{D}_{x_n}\widetilde{D}_{x_m}B+\widetilde{D}_{x_m}B^{P_l}u_{ln}+\widetilde{D}_{x_n}B^{P_k}u_{km}+B^{P_kP_l}u_{ln}u_{km},
    \end{split}
\end{equation}
where the $B^{P_k}$ and $B^{P_kP_l}$ denote $D_{p_k}B$ and $D_{p_kp_l}B$ and sometimes hereinafter referred to as $B^k$ and $B^{kl}$.

Last but not least, For any function $g(x,z,p)\in C^2(\overline{\Omega}\times \mathbb{R}\times  \mathbb{R}^n)$, we analyze the expression for $Lg(x,u,Du)$, where $L$ is the operator defined in (2.4). We calculate the derivatives of $g(x,u,Du)$.

\begin{equation}
    \tag{2.9}
    \begin{split}
    g_i=&D_{x_i}g+D_zgu_i+D_{p_k}gu_{ki},\\
    g_{ij}=&D_{x_i}D_{x_j}g+D_{x_i}D_{z}gu_j+D_{x_i}D_{p_l}gu_{lj}+D_{x_j}D_zgu_i+D_zD_zgu_iu_j+D_{pl}D_zgu_iu_{lj}+\\
    &D_zgu_{ij}+D_{x_j}D_{p_k}gu_{ki}+D_zD_{p_k}gu_ju_{ki}+D_{p_l}D_{p_k}gu_{ki}u_{lj}+D_{p_k}gu_{kij},\\
    g_{ii}=&D_{x_i}D_{x_i}g+2D_{x_i}D_{z}gu_i+2D_{x_i}D_{p_k}gu_{ki}+D_zD_zgu_i^2+2D_{pk}D_zgu_iu_{ki}+D_zgu_{ii}+\\
    &D_{p_l}D_{p_k}gu_{ki}u_{li}+D_{p_k}gu_{kii},\\
    \Delta g= &\Delta_xg+2\sum_iD_{x_i}D_{z}gu_i+2\sum_iD_{x_i}D_{p_k}gu_{ki}+D_zD_zg\sum_iu_i^2+2\sum_iD_{pk}D_zgu_iu_{ki}+\\
    &D_zg\Delta u+\sum_iD_{p_l}D_{p_k}gu_{ki}u_{li}+D_{p_k}g\Delta u_{k},
    \end{split}
\end{equation}
where $\Delta_xg=\sum_iD_{x_i}D_{x_i}g$. Next we use (2.9) to calculate $Lg$. It has
\begin{equation}
\tag{2.10}
\begin{split}
        Lg=&D_zgLu+D_{p_k}gLu_k+F^{ij}[\gamma\delta_{ij}\sum_mD_{p_l}D_{p_k}gu_{km}u_{lm}-D_{p_l}D_{p_k}gu_{ki}u_{lj}]+\\
        &F^{ij}[\gamma\delta_{ij}\Delta_xg-D_{x_i}D_{x_j}g-A_{ij}^{p_k}D_{x_k}g]+F^{ij}[\gamma\delta_{ij}2\sum_m\widetilde{D}_{x_m}D_{p_k}gu_{km}-2\widetilde{D}_{x_i}D_{p_k}gu_{jk}]+\\
        &F^{ij}[\gamma\delta_{ij}2\sum_m\widetilde{D}_{x_m}D_zgu_m-\gamma\delta_{ij}D_zD_zg\sum_mu_m^2-2\widetilde{D}_{x_i}D_zgu_j+D_zD_zgu_iu_j],
\end{split}
\end{equation}
where 
\begin{equation}
    \tag{2.11}
    \begin{split}
    &Lu=F^{ij}(\gamma\delta_{ij}\Delta u-\delta_{ik}u_{jk}-A_{ij}^{p_k}u_k),\\
   &Lu_k=F^{ij}\widetilde{D}_{x_k}A_{ij}+\widetilde{D}_{x_k}B+D_{p_l}Bu_{kl}.
    \end{split}
\end{equation}

Assume that $g$ is quasilinear, then we have $D_{p_i}D_{p_j}g=0$, for $ \forall 1\le i,j\le n$. And we define
\begin{equation}
\notag
\widetilde{\beta}_{ik}:=2\widetilde{D}_{x_i}D_{p_k}g+D_zg\delta_{ik}.
\end{equation}
Plugging (2.11) into (2.10) and considering (1.8) , we obtain
\begin{equation}
    \tag{2.12}
    \begin{split}
    Lg=&D_zg(\widetilde{F}^{ij}\delta_{ik}u_{jk}-F^{ij}A_{ij}^{p_k}u_k)+D_{p_k}g(F^{ij}\widetilde{D}_{x_k}A_{ij}+\widetilde{D}_{x_k}B+D_{p_l}Bu_{kl})+\\
    &F^{ij}[\gamma\delta_{ij}\Delta_xg-D_{x_i}D_{x_j}g-A_{ij}^{p_k}D_{x_k}g]+2\widetilde{F}^{ij}\widetilde{D}_{x_i}D_{p_k}gu_{jk}+\\
    &F^{ij}[\gamma\delta_{ij}2\sum_m\widetilde{D}_{x_m}D_zgu_m-\gamma\delta_{ij}D_zD_zg\sum_mu_m^2-2\widetilde{D}_{x_i}D_zgu_j+D_zD_zgu_iu_j].
    \end{split}
\end{equation}
Using assumption $F4$, then it follows that a inequality about $Lg$, which is

\begin{equation}
    \tag{2.13}
    Lg\ge  -C\mathcal{F}+\widetilde{F}^{ij}\widetilde{\beta}_{ik}u_{jk}+D_{p_k}gD_{p_l}Bu_{kl},
\end{equation}
where $C$ is a constant depending on $n,\Omega,||A||_{C^1},||B||_{C^1}$ and $||u||_{C^1;\Omega}$. 

There is relation about $\mathrm{trace}(\widetilde{F}^{ij})$ and $\mathcal{F}$, which is
\begin{equation}
    \tag{2.14}
    \mathrm{trace}(\widetilde{F}^{ij})=(n\gamma-1)\mathcal{F}.
\end{equation}
(2.14) will be used several times in our article. 

All the preparatory work is now complete. Next, we will begin the proof chapter.

\section{Second derivative Estimates}

The proof of the second derivative estimates will be divided into five parts. The first part deals with the second derivative estimate for admissible solutions in the interior of the domain, while the second through fourth parts concern the second derivative estimate for admissible solutions on the boundary of the domain. The last part is the proof of Theorem 1.2.

\subsection{Interior second derivative estimate}
\begin{proof}

 Similarly to (2.7) and (2.8), we can simplify the formulas $\mathcal{L}u_{\tau},\mathcal{L}u_{\tau\tau}$ by differentiating  $F(\gamma\Delta uI-D^2u-A(x,u,Du))=B (x,u,Du)$ along the vector $\tau$ once and twice. A direct computation yields
\begin{equation}
    \tag{3.1}
    \begin{split}
\mathcal{L}u_\tau= &F^{ij}\widetilde{D}_{x_\tau}A_{ij}+\widetilde{D}_{x_\tau}B,\\
 \mathcal{L}u_{\tau\tau}=&-F^{ij，kl}(w_{kl})_\tau(w_{ij})_\tau+F^{ij}(\tau^m\tau^n\widetilde{D}_{x_n}\widetilde{D}_{x_m}A_{ij}+2\widetilde{D}_{x_\tau}A_{ij}^ku_{k\tau}+A_{ij}^{kl}u_{l\tau}u_{k\tau})+\\
&\tau^m\tau^n\widetilde{D}_{x_n}\widetilde{D}_{x_m}B+2\widetilde{D}_{x_\tau}B^{P_k}u_{k\tau}+B^{P_kP_l}u_{l\tau}u_{k\tau}.
    \end{split}
\end{equation}

Now we can see, the main difficulty in (3.1) is the term $F^{ij}A_{ij}^{kl}u_{l\tau}u_{k\tau}$. First we estimate it as
\begin{equation}
    \tag{3.2}
   F^{ij}A_{ij}^{kl}u_{\tau l}u_{\tau k}\ge-CF^{ij}|u_{\tau l}||u_{\tau k}|\ge -C\mathcal{F}|u_{\tau l}||u_{\tau k}|\ge-C\mathcal{F}\sum_mu_{\tau m}^2 ,
\end{equation}
where the constant $C$ depends on $n,||A||_{C^2}$ and $||u||_{C^1;\Omega}$. We can use $L|u_\tau|^2$ or $\mathcal{L}|u_\tau|^2$ to overcome this difficulty.

We first calculate $L|u_\tau|^2$.
\begin{equation}
    \tag{3.3}
    L|u_\tau|^2=F^{ij}(\gamma\delta_{ij}\Delta (|u_\tau|^2)-(|u_{\tau}|^2)_{ij}-A^{P_k}D_k(|u_\tau|^2)).
\end{equation}
It is a bit involved to calculate $\Delta (|u_\tau|^2)$. We write it out in detail.
\begin{equation}
    \tag{3.4}
    \begin{split}
        &(|u_\tau|^2)_i=2u_\tau u_{ki}\tau^k,\\
        &(|u_\tau|^2)_{ij}=2u_{\tau i}u_{\tau j}+2u_{\tau}u_{kij}\tau^k,\\
        &(|u_\tau|^2)_{ii}=2u_{\tau i}^2+2u_{\tau}u_{ii\tau},\\
        \Delta &(|u_\tau|^2)=\sum_i2u_{\tau i}^2+2u_\tau(\Delta u)_\tau.
    \end{split}
\end{equation}
Then one can get the formulas of $L|u_\tau|^2$ and $\mathcal{L}|u_\tau|^2$,
\begin{equation}
    \tag{3.5}
    \begin{split}
        \frac{c}{2}L|u_\tau|^2=&cF^{ij}[\gamma\delta_{ij}(\sum_mu_{\tau m}^2+u_\tau(\Delta u)_\tau)-u_{\tau i}u_{\tau j}-u_{\tau}u_{kij}\tau^k-A_{ij}^ku_\tau u_{\tau k}]\\
        &=cF^{ij}[\gamma\delta_{ij}\sum_mu_{\tau m}^2-u_{\tau i}u_{\tau j}]+cu_\tau Lu_\tau,\\
        \frac{c}{2}\mathcal{L}|u_\tau|^2=&cF^{ij}[\gamma\delta_{ij}\sum_mu_{\tau m}^2+u_\tau(\Delta u)_\tau)-u_{\tau i}u_{\tau j}-u_{\tau}u_{kij}\tau^k-A_{ij}^ku_\tau u_{\tau k}]-cB^ku_\tau u_{\tau k}\\=
        &cF^{ij}[\gamma\delta_{ij}\sum_mu_{\tau m}^2-u_{\tau i}u_{\tau j}]+cu_\tau\mathcal{L}u_\tau,
    \end{split}
\end{equation}

Using Lemma 2.1 to scale $ \frac{c}{2}\mathcal{L}|u_\tau|^2$. Choosing the constant $c$ large enough, it can show
\begin{equation}
    \tag{3.6}
    \frac{c}{2}\mathcal{L}|u_\tau|^2 \ge c(\gamma-1)\mathcal{F}\sum_mu_{\tau m}^2+cu_\tau\mathcal{L}u_{\tau}\ge 2C\mathcal{F}\sum_mu_{\tau m}^2+cu_\tau\mathcal{L}u_{\tau},
\end{equation}
where the constant $c$ only depends on $n,||A||_{C^2},\gamma$ and $||u||_{C^1;\Omega}$. That means we can choose a proper constant $c$ to achieve our goal to eliminate $F^{ij}A_{ij}^{kl}u_{l\tau}u_{k\tau}$ in (3.1). Add (3.1) and (3.5) together, Using the assumption $F2$ this yields
\begin{equation}
    \tag{3.7}
    \begin{split}
    \mathcal{L}(u_{\tau\tau}+\frac{c}{2}|u_\tau|^2)\ge& F^{ij}(\tau^m\tau^n\widetilde{D}_{x_n}\widetilde{D}_{x_m}A_{ij}+2\widetilde{D}_{x_\tau}A_{ij}^ku_{k\tau})+C\mathcal{F}\sum_mu_{\tau m}^2+cu_\tau\mathcal{L}u_{\tau}\\
    &\tau^m\tau^n\widetilde{D}_{x_n}\widetilde{D}_{x_m}B+2\widetilde{D}_{x_\tau}B^{P_k}u_{k\tau}+B^{P_kP_l}u_{l\tau}u_{k\tau}\\
   &\ge C\mathcal{F}|Du_\tau|^2-C(1+\mathcal{F})(1+|Du_\tau|)+B^{P_kP_l}u_{l\tau}u_{k\tau}\\
   &\ge C\mathcal{F}|Du_\tau|^2-C(1+\mathcal{F})(1+|Du_\tau|)+\lambda_B|Du_\tau|^2,
\end{split}
\end{equation}
where the constant $C$ depends on $n,||A||_{C^2;\Omega},||B||_{C^2;\Omega},\gamma$ and $||u||_{C^1;\Omega}$, $\lambda_B$ is the minimum eigenvalue of matrix $D^2_pB$. If the assumption (ii) (B is convex with respect to variable $p$) in Theorem 1.1 holds, this yields 
\begin{equation}
    \tag{3.8}
     \mathcal{L}(u_{\tau\tau}+\frac{c}{2}|u_\tau|^2)\ge C\mathcal{F}|Du_\tau|^2-C(1+\mathcal{F})(1+|Du_\tau|).
\end{equation}

For the assumption (i) ($F4^+$ holds) in Theorem 1.1. Since $M[u]=\gamma\Delta uI-D^2u-A$ is linearly invertible and $A$ is bounded, if $|D^2u|$ is large enough, it can be shown that $|M[u]|$ is large enough. So we can use assumption $F4^+$ to get $C\mathcal{F}+\lambda_B>\frac{C\mathcal{F}}{2}$ when $|Du_\tau|^2$ is large enough. (3.8) can also be derived under the assumption i.

Finally, we construct the classical auxiliary function $e^{ax_1}$, where $a\in \mathbb{R}$ is constant. Add it into $u_{\tau\tau}+\frac{c}{2}|u_\tau|^2$ and applying $\mathcal{L}$, using Young's Inequality, we obtain
\begin{equation}
    \tag{3.9}
    \begin{split}
          \mathcal{L}(u_{\tau\tau}-\frac{C_1}{2}|u_\tau|^2+e^{ax_1})&\ge[C\mathcal{F}-\epsilon(1+\mathcal{F})]|Du_\tau|^2-C(\epsilon)(1+\mathcal{F})+\\
          &e^{ax_1}[F^{11}(\gamma-1)a^2+\sum_{k=2}^nF^{kk}a^2-aF^{ij}A_{ij}^{p_1}-aB^{p_1}],
    \end{split}
\end{equation}
where $\epsilon$ is a small constant. Since $\Omega$ is boundeded, we can find a constant $R\in\mathbb{R}$ such that $x_1+R>0$ for any $x\in\Omega$, So using the assumption $F4$ or $F4^+$, Choosing $a$ large enough, it follows that
\begin{equation}
    \tag{3.10}
     \mathcal{L}(u_{\tau\tau}-\frac{C_1}{2}|u_\tau|^2+e^{a(x_1+R)})\ge[C\mathcal{F}-\epsilon(1+\mathcal{F})]|Du_\tau|^2+[-C(\epsilon)+a^2(\gamma-1)-aC]\mathcal{F}\ge0.
\end{equation}
The key point is to choose $\epsilon$ small enough to make sure $C\mathcal{F}-\epsilon(1+\mathcal{F})>0$ and let $a$ large enough to make sure $a^2(\gamma-1)-C(\epsilon)-aC>0$.

Then, by the classical maximum principle, under the hypotheses of Theorem1.1, we have
\begin{equation}
    \tag{3.11}
    \underset{\Omega}{\sup}\ u_{\tau\tau}\le\underset{\partial\Omega}{\sup}\ u_{\tau\tau}+C,
\end{equation}
where constant $C$ depends on the same quantities as in (3.7). This yields a global upper bound for $D^2u$, since $\tau$ can be any unit vector.

For the lower bounds, we shall use the assumption $\Gamma\subset\Gamma_1^+$. If $u$ is an admissible solution for (1.1), that means
\begin{equation}
    \tag{3.12}
     \gamma\Delta uI-D^2u-A\in\Gamma_1\to(n\gamma-1)\Delta u>\mathrm{tr}(A)\ge-C,
\end{equation}
where the constant $C$ depends on $n,\Gamma$ and $||A||_{C^0;\Omega}$. 

Since (3.11), we know that the second-order directional derivatives of the admissible solution $u$ in all directions are bounded above, by direct calculation it follows that 
\begin{equation}
\tag{3.13}
    \lambda_{\max}(D^2u)\le\underset{\partial\Omega}{\sup}\ u_{\tau\tau}+C,
\end{equation}
where $\lambda_{\max}(D^2u)$ means the largest eigenvalue of the Hessian matrix $D^2u$ and constant $C$ depends on the same quantities as in (3.11). Without loss of generality, let the smallest eigenvalue of the matrix $D^2u$ be $\lambda_1(D^2u)$, it can be shown that

\begin{equation}
    \tag{3.14}
    \lambda_{\min}(D^2u)=\Delta u-\sum_{i=2}^n\lambda_i(D^2u)\ge-(n-1)\underset{\partial\Omega}{\sup}\ u_{\tau\tau}-C,
\end{equation}
where $\lambda_i(D^2u)$ means the eigenvalue of the matrix $D^2u$ and constant $C$ depends on the same quantities as in (3.11) and additionally $\Gamma$. Combining (3.13) and (3.14), for any $D_{ij}u,1\le i,j\le n$, a direct calculation yields 
\begin{equation}
    \tag{3.15}
    |D_{ij}u|\le n\, \underset{\partial\Omega}{\sup}\ u_{\tau\tau}+C,
\end{equation}
where the constant $C$ depends on $n,A,B,\gamma,F,\Gamma,\Omega$ and $||u||_{C^1;\Omega}$.
\end{proof}

\subsection{Boundary second derivative estimate for mixed tangential-oblique direction}
To prove Theorem 1.2, we have to establish estimates for second derivatives on the boundary $\partial\Omega$ under the boundary condition (1.2). Before we begin the estimate, for convenience, we assume the functions $G(x,u,Du)$ and unit normal vector $\nu$ have been extended to $\overline{\Omega}$ properly. 

We divide it into three parts and estimate them separately. They are mixed tangential-oblique direction, double oblique direction and double tangential direction. First and the simplest part is mixed tangential-oblique direction. We differentiate the boundary condition function $G(x,u,Du)$ along the tangential direction $\tau$ at any point $x\in\partial\Omega$. Since $G\equiv0$ on $\partial\Omega$, so the derivative is also zero.
\begin{equation}
    \tag{3.16}
    \widetilde{D}_{x_\tau}G+G_{p_l}u_{lm}\tau_m=0, \,  \mathrm{on} \    \partial\Omega,
\end{equation}
which means 
\begin{equation}
    \tag{3.17}
    u_{\beta\tau}=-\widetilde{D}_{x_\tau}G\to \ \ |u_{\beta\tau}| \le C, \ \ \mathrm{on} \ \   \partial\Omega,
\end{equation}
for any unit tangential vector field $\tau$. Where $\beta=D_pG(x,u,Du)$ and $C$ is a constant depending on $G,\Omega$ and $||u||_{C^1;\Omega}$.
So we have the estimate for second derivative in tangential-$\beta$ direction on boundary.

\subsection{Boundary second derivative estimate for double oblique direction}
Before we start this part, let's give some notations. Setting
\begin{equation}
    \notag
    M_2:=\sup_\Omega|D^2u|, \ \ M_2':=\sup_{\partial\Omega}\sup_{|\tau=1|,\tau\cdot\nu=0}|u_{\tau\tau}|.
\end{equation}
Then we state the result for this part.

\begin{lemma}
Let $u\in C^3(\overline{\Omega})$ be an admissible solution of the boundary value problem $\mathrm{(1.1)}$ in a bounded $C^2$ domain $\Omega\subset \mathbb{R}^n$, without any convexity assumption for $\Omega$. Assume the operator $F$ satisfies conditions $F1-F4$, $A\in C^1(\overline{\Omega}\times \mathbb{R}\times \mathbb{R}^n)$, $B\in C^1(\overline{\Omega}\times \mathbb{R}\times \mathbb{R}^n)$ satisfies $B>a_0$ on the range of $(x,u,Du)$. and $G\in C^2(\partial\Omega\times \mathbb{R}\times \mathbb{R}^n)$ is oblique with respect to $u$ satisfying (1.2), either $F4^+$ holds or $B$ is independent of variable $p$ and $\mathcal{G}$ is quasilinear satisfying (1.11). Then for any $\epsilon>0$, 
\begin{equation}
    \tag{3.18}
    \sup_{\partial\Omega}u_{\beta\beta}\le\epsilon M_2+C(\epsilon)(1+M_2'),
\end{equation}
where $\beta=D_pG(x,u,Du)$ and constant $C(\epsilon)$ depends on $\epsilon,F,A,B,\gamma,G,\Omega,\beta_0$ and $||u||_{C^1;\Omega}$.
\end{lemma}

\begin{proof}

First we construct an auxiliary function $f$, for any point $x_0\in\partial\Omega$, we define
\begin{equation}
    \tag{3.19}
    f(x)=\frac{a}{2}|Du-Du(x_0)|^2,
\end{equation}
where $a\le1$ is a positive constant. Then we consider the expression of $Lf$. First
\begin{equation}
    \tag{3.20}
    \begin{split}
    &f_i=a\sum_k(u_k-u_k(x_0))u_{ki},\\    
    &f_{ij}=a\sum_ku_{ki}u_{kj}+a\sum_k(u_k-u_k(x_0))u_{kij},\\
    &\Delta f=a\sum_{k,m}u_{km}^2+a\sum_k(u_k-u_k(x_0))\Delta u_{k}.
    \end{split}
    \end{equation}
So we have
\begin{equation}
    \tag{3.21}
    L(\frac{a}{2}|Du-Du(x_0)|^2)=aF^{ij}(\gamma\delta_{ij}u_{km}^2-u_{kj}u_{ki})+a(u_k-u_k(x_0))Lu_k=a\widetilde{F}^{ij}u_{ki}u_{kj}+a(u_k-u_k(x_0))Lu_k.
\end{equation}

Then we define the second auxiliary function $\overline{v}$, it is
\begin{equation}
    \tag{3.22}
    \overline{v}=G(x,u,Du)+\frac{a}{2}|Du-Du(x_0)|^2,
\end{equation}
where $G$ is the boundary function in (1.1).

 By (2.13), (2.14), (3.21), Young's inequality and Cauchy inequality, calculating $L\overline{v}$, we obtain 
 \begin{equation}
    \tag{3.23}
    \begin{split}
    L\overline{v}\ge&-C\mathcal{F}+\widetilde{F}^{ij}\widetilde{\beta}_{ik}u_{jk}+(D_{p_k}G)(D_{p_l}B)u_{kl}+a\widetilde{F}^{ij}u_{ki}u_{kj}+a(u_k-u_k(x_0))Lu_k.\\
    \ge&-C(1+a)\mathcal{F}-\frac{1}{a}\widetilde{F}^{ij}\widetilde{\beta}_{ik}\widetilde{\beta}_{jk}+[D_{p_k}G+a(u_k-u_k(x_0))](D_{p_l}B)u_{kl}\\
    \ge&-\mathcal{F}(\frac{C}{a}+\epsilon_1M_2+C_{\epsilon_1}),
    \end{split}
    \end{equation}
for any $\epsilon_1>0$, where $g$ is replaced by $G$ in $\widetilde{\beta}_{ik}$. $F4^+$ is used in the last inequality when $|D^2u|$ is large enough. The constant $C$ and $C_{\epsilon_1}$ depend on $n,\Omega,\gamma,||G||_{C^2},||A||_{C^1},||B||_{C^1},F,\epsilon_1$ and $||u||_{C^1}$.
In the estimation (3.23), The following relation is used
\begin{equation}
    \tag{3.24}
         |D^2u(x)|\le\epsilon_1\mathcal{F}(x)M_2+C_{\epsilon_1}\mathcal{F}(x) 
\end{equation}
for any $\epsilon_1>0$, it can be justified under conditions $F4$ and $F4^+$, by distinguishing the following two cases according to the size of the augmented Hessian $|M[u]|$.

(i) Bounded case: $|M[u]|\le C_0$. From $F4$ we know $\mathcal{F}(x)$ has positive lower bound and then it can proved directly. 

(ii) Large-Hessian case: $|M[u]|$ is sufficiently large that $\epsilon_1\mathcal{F}(x)>1$ from $F4^+$. In this case we have $|D^2u(x)|\le\epsilon_1\mathcal{F}(x)M_2$.

Next we construct an upper barrier for $\overline{v}$ at the point $x_0$. So we define the third auxiliary function $\phi$, 
\begin{equation}
    \tag{3.25}
    \phi=c(d-td^2),
\end{equation}
where $d=d(x)=\mathrm{dist}(x,\partial\Omega)$ is distance function. $c,t$ are positive constants to be determined. It is well known that, If the domain $\Omega$ is a $C^2$ domain, then exists a small constant $\rho>0$, such that the function $d(x)$ is $C^2$ in $\Omega_\rho:=\{x\in\Omega|d(x)<\rho\}$. In this way, we can apply the operator $L$,
as defined above, to $\phi$ on the domain $\Omega_\rho$. First 
\begin{equation}
\tag{3.26}
\begin{split}
    &\phi_i=c(1-2td)d_i \ \ \phi_{ij}=c(1-2td)d_{ij}-2tcd_id_j,\\
    &\Delta\phi=c(1-2td)\Delta d-2tc|Dd|^2.
\end{split}
\end{equation}
Then applying operator $L$, it follows that
\begin{equation}
\tag{3.27}
\begin{split}
    L\phi&=cF^{ij}[\gamma\delta_{ij}((1-2td)\Delta d-2t|Dd|^2)-(1-2td)d_{ij}+2td_id_j-A_{ij}^{p_k}(1-2td)d_k],\\
    \frac{1}{c(2td-1)} L\phi&=F^{ij}[\gamma\delta_{ij}(-\Delta d)+d_{ij}+A_{ij}^{p_k}d_k]+\frac{2t}{1-2td}F^{ij}(\gamma\delta_{ij}|Dd|^2-d_id_j).
\end{split}
\end{equation}

As we know, $|Dd|\equiv1$ on $\partial\Omega$. So we can let $\rho>0$ small enough such that $|Dd|>\frac{1}{2}$ on $\Omega_\rho$. Then we let $t$ not so large to satisfy $t\rho\le\frac{1}{4}$. Using Lemma 2.1, we can scale $\frac{1}{c(2td-1)} L\phi$ as  
\begin{equation}
    \tag{3.28}
    \frac{1}{c(2td-1)} L\phi\ge F^{ij}[\gamma\delta_{ij}(-\Delta d)+d_{ij}+A_{ij}^{p_k}d_k]+2t(\gamma-1)F^{ij}\delta_{ij}|Dd|^2\ge-C\mathcal{F}+t(\gamma-1)\mathcal{F}\ge2\mathcal{F}.
\end{equation}
Here, the constant $C$ depends on the curvature of the domain, but don't need to assume any convexity condition on the domain. To get the last inequality, we must let $t$ is large enough. In order to keep the restriction $t\rho\le\frac{1}{4}$, we must let $\rho>0$ small enough. Therefore we have
\begin{equation}
    \notag
     L\phi\le -c\mathcal{F}, \,  \mathrm{in}  \ \partial\Omega_\rho.
\end{equation}

Then the forth and the last auxiliary function is be defined as $\overline{\phi}$, 
\begin{equation}
    \tag{3.29}
    \overline{\phi}=\phi+b|x-x_0|^2,
\end{equation}
where $b$ is a constant to be determined later. Applying $L$   to $\overline{\phi}$, it can be shown
\begin{equation}
    \tag{3.30}
    L\overline{\phi}\le -c\mathcal{F}+Cb\mathcal{F}, \,  \mathrm{on} \  \Omega_\rho,
\end{equation}
where the constant $C$ depends on $n,||A||_{C^1},\Omega$ and $\gamma$. Provided $c\ge C(b+\frac{1}{a})+\epsilon_1 M_2+C_{\epsilon}$, we get
\begin{equation}
    \tag{3.31}
    L\overline{\phi}\le L\overline{v} \to L(\overline{v}-\overline{\phi})\ge0, \, \mathrm{in} \  \Omega_\rho. 
\end{equation}

Next, we need to examine $\overline{v}$ and $\overline{\phi}$ on the boundary of $\Omega_\rho$. 

For any $x\in\partial\Omega$, $\overline{v}=\frac{a}{2}|Du-Du(x_0)|^2$ and $\overline{\phi}=\frac{b}{2}|x-x_0|^2$, then 
\begin{equation}
    \tag{3.32}
    |Du-Du(x_0)|^2\le |C(1+M_2')|x-x_0||^2,
\end{equation}
where the strict obliqueness (1.2) and (3.17) are used, so the constant $C$ also depends on $\beta_0$.

For any $x\in\partial\Omega_\rho\cap\Omega$,  $\overline{v}=G+\frac{a}{2}|Du-Du(x_0)|^2\le C$ and  $\overline{\phi}=\phi+\frac{b}{2}|x-x_0|^2\ge\frac{3}{4}c\rho$, where the constant $C$ depends on $G$ and $||u||_{C^1}$. By choosing $b=aC[1+(M_2^{'})^2]$ and $c\ge \frac{4C}{3\rho}$, it follows that 
\begin{equation}
    \tag{3.33}
    \overline{v}\le\overline{\phi}, \,  \mathrm{in} \  \Omega_\rho.
\end{equation}

Since $ \overline{v}(x_0)=\overline{\phi}(x_0)$, we know
\begin{equation}
\tag{3.34}
    D_\beta\overline{v}(x_0)\le D_\beta\overline{\phi}(x_0),
\end{equation}
which implies
\begin{equation}
    \tag{3.35}
    u_{\beta\beta}(x_0)\le\beta^0c+C,
\end{equation}
where $\beta^0:=\underset{\partial\Omega}{\sup}(G_p(x,u,Du)\cdot\nu)\ge\beta_0$. Then for any $\epsilon>0$, we can take $a=\frac{1}{1+\epsilon_1M_2}$ and $\epsilon_1=\frac{\epsilon}{\beta^0C}$, it has  
\begin{equation}
    \tag{3.36}
    \begin{split}
    c\le& 2[C(b+\frac{1}{a})+\epsilon_1 M_2+C_{\epsilon}]+\frac{4C}{3\rho}\\
    \le& C[a(M_2')^2+(a+\epsilon_1)M_2+\frac{1}{a}]\\
    \le&C[\frac{1}{1+\epsilon_1M_2}(M_2')^2+(\epsilon_1+\frac{1}{1+\epsilon_1M_2})M_2+1+\epsilon_1M_2]\\
    \le& C[C_\epsilon (M_2')^2+\frac{\epsilon}{\beta^0C}M_2+C_\epsilon],
    \end{split}
\end{equation}
where constants $C$ and $C_\epsilon$ depends on $\gamma,F,A,B,G,\Omega,\beta_0,\epsilon$ and $||u||_{C^1;\Omega}$.

Combining (3.36) and (3.35), we obtain

\begin{equation}
    \tag{3.37}
     u_{\beta\beta}(x_0)\le \epsilon M_{2}+C_\epsilon (1+M_2').
\end{equation}
 Since $x_0$ is an arbitrary point on $\partial\Omega$, we can take the supremum of (3.37) on $\partial\Omega$ to get the estimate (3.18) in Lemma3.1. Therefore, the case $F4^+$ holds has been proved. While only $F4$ holds and $B$ independent of variable $p$, the last term in third line of (3.23) will not appear. So the same estimate (3.18) and Lemma 3.1 are thus proved in the same way.

 \end{proof}

\subsection{Boundary second derivative estimate for double tangential direction}
Now only purely tangential derivatives on the boundary need to be estimated. In this part, we formulate it as follows.

\begin{lemma}
    Let $u\in C^2(\overline{\Omega})\cap C^4(\Omega)$ be an admissible solution of the boundary value problem $\mathrm{(1.1)}$ in a bounded $C^2$ domain $\Omega\subset \mathbb{R}^n$. Assume that $F,A$ and $B$ satisfy the hypothesis of Theorem 1.1 and $G\in C^2(\partial\Omega\times \mathbb{R}\times\mathbb{R}^n)$ is quasilinear and oblique with respect to $u$ satisfying $\mathrm{(1.2)}$.Then for any tangential vector field $\tau$, $|\tau|\le1$ and constant $\epsilon>0$, we have estimate
    \begin{equation}
        \tag{3.38}
        M_2^+(\tau)\le\epsilon M_2+C_\epsilon,
    \end{equation}
where $M_2^+=\underset{\partial\Omega}{\sup}u_{\tau\tau}$, and constant $C_\epsilon$ depends on $\epsilon,\gamma,A,B,G,\Omega,\beta_0$ and $||u||_{C^1;\Omega}$. 
\end{lemma}

\begin{proof}

Extending unit inner normal vector $\nu$ and $G$ to $\overline{\Omega}$ and in a neighborhood of $\partial\Omega$, preserving their constancy along the normal direction. Construct auxiliary function
\begin{equation}
    \tag{3.39}
    v_\tau=u_{\tau\tau}+\frac{c}{2}|u_\tau|^2.
\end{equation}
Assume that on $\partial\Omega$ and for tangent vectors $|\tau|\le1$, the maximum is attained at $x_0=0$ and $\tau_0=e_1$.
Construct
\begin{equation}
    \tag{3.40}
     b=\frac{\nu_1}{\beta\cdot\nu},\ \ \tau=e_1-b\beta.
\end{equation}

In this section, we define $u_\tau,v_\tau:=\sum_iu_i\tau_i,\sum_iv_i\tau_i$, where $\tau$ is not necessarily a unit vector. A direct computation shows:
\begin{equation}
\tag{3.41}
\begin{split}
   &v_1=v_\tau+b(2u_{\beta\tau}+cu_\beta u_\tau)+b^2(u_{\beta\beta}+\frac{c}{2}u_\beta^2), \\
   &v_1(0)=v_\tau(0),\ \ b(0)=0,\ \ \tau(0)=e_1.
\end{split}
\end{equation}

Setting $g:=\frac{1}{\beta\cdot\nu}(2u_{\beta\tau}+cu_\beta u_\tau)$, Using (3.17) and mean value theorem, this yields 
\begin{equation}
    \notag
   |g(x)-g(0)|\le C(1+M_2)|x|,\ \ |\nu_1(x)|=|\nu_1(x)-\nu_1(0)|\le C|x|, \,  \mathrm{on}  \ \partial\Omega.
\end{equation}
where the constant $C$ depends on $G,\Omega,\beta_0$ and $||u||_{C^1}$.

Such that the function $\widetilde{v}_1$ defined below satisfies
\begin{equation}
    \tag{3.42}
    \begin{split}
    &\widetilde{v}_1=v_1-g(0)\nu_1-C_1(1+M_2)|x|^2,\\
    &\widetilde{v}_1\le v_\tau\le|\tau|^2\frac{v_\tau}{|\tau|^2}\le|\tau|^2v_1(0)=|\tau|^2\widetilde{v_1}(0), \,  \mathrm{on} \ \partial\Omega,
    \end{split}
\end{equation}
where the constant $C_1$ depends on the same quantities as $C$.

Since $G$ is quasilinear ($\beta=\beta(x,u)$), we can estimate $|\tau|^2 $ as
\begin{equation}
    \tag{3.43}
    |\tau|^2\le1-2b_1\beta_1+b^2\sup_\Omega|\beta|^2\le1-2\frac{\beta_1}{\beta\cdot\nu}(0)\nu_1+C_1|x|^2, \ \ \mathrm{on} \ \  \Omega.
\end{equation}

So set $f:=1-2\frac{\beta_1}{\beta\cdot\nu}(0)\nu_1+C_1|x|^2$, $f(0)=1$ and then it has
\begin{equation}
    \notag
    \widetilde{v}_1\le f \widetilde{v}_1(0), \,  \mathrm{on}  \ \partial\Omega.
\end{equation}

Next, we differentiate $G=\sum_i\beta_i(x,u)u_i-\varphi(x,u)$ along the tangential direction $\tau$ at point $0\in\partial\Omega$, up to second order. It has
\begin{equation}
    \tag{3.44}
    \begin{split}
    G_\tau=&(G_{x_{i}}+G_zu_i+G_{p_k}u_{ki})\tau_i=0,\\
    G_{\tau\tau}=&(G_{x_{i}x_{i}}+G_{x_{i}z}u_j+G_{x_{i}p_l}u_{lj}+G_{zx_j}u_i+G_{zz}u_iu_j+G_{zp_l}u_iu_{jl}+G_zu_{ij}+\\
    &G_{x_jp_k}u_{ki}+G_{zp_k}u_ju_{ki}+G_{p_lp_k}u_{lj}u_{ki}+G_{p_k}u_{kij})\tau_i\tau_j=0.
    \end{split}
\end{equation}
Since $\tau$ is a fixed tangential direction at $0$, we may extend it to a constant vector field in a neighborhood of $0$. Differentiating $G=0$ along this constant field and evaluating at $0$ yields the desired tangential derivative identities.

Substituting $\tau(0)=e_1$ into (3.44), it becomes 
\begin{equation}
    \tag{3.45}
    \begin{split}
         G_{11}=&G_{x_{1}x_{1}}+G_{x_{1}z}u_1+G_{x_{1}p_l}u_{l1}+G_{zx_1}u_1+G_{zz}u_1u_1+G_{zp_l}u_1u_{1l}+G_zu_{11}+\\
    &G_{x_1p_k}u_{k1}+G_{zp_k}u_1u_{k1}+G_{p_lp_k}u_{l1}u_{k1}+G_{p_k}u_{k11}=0.
    \end{split}
\end{equation}
Since $G$ is quasilinear, $D^2_pG\equiv0$. So it can be shown that
\begin{equation}
    \tag{3.46}
    G_{p_k}u_{k11}(0)=D_\beta u_{11}(0)\ge-C(1+M_2),
\end{equation}
where $C$ is a constant depending on $||G||_{C^2}$ and $||u||_{C^1}$.

Moreover, we construct the auxiliary function
\begin{equation}
    \tag{3.47}
    v=\widetilde{v}_1-f\widetilde{v}_1(0)-K(1+M_2)\phi,
\end{equation}
where $K$ is a constant depending on the same quantities as $C_1$ and $\phi\in C^2(\overline{\Omega})$ is a negative defining function satisfying $\phi\equiv0$ on $\partial\Omega$, $D_\nu\phi=-1$ on $\partial\Omega$. For $K$ sufficiently large, we can conclude that $v$ attains its maximum at an interior point $x_0\in\Omega$. Otherwise, it easy to see $v$ attains its maximum at $0$ on $\overline{\Omega}$, and we can use (3.46), $D_\nu\phi=1$ on $\partial\Omega$ and strict obliqueness (1.2) to prove $D_\beta v(0)>0$, it's a contradiction. Applying $L$ defined in (2.4) to $v$ at the point $x_0$, by(1.8), we obtain
\begin{equation}
    \tag{3.48}
    \mathcal{L}v(x_0)=F^{ij}(\gamma\delta_{ij}\Delta v-v_{ij}-A_{ij}^{P_k}D_kv)-D_{P_k}B(x,u,Du)D_kv=F^{ij}(\gamma\delta_{ij}\Delta v-v_{ij})=\widetilde{F}^{ij}v_{ij}\le0.
\end{equation}
Using the result in (3.8) and calculating $\mathcal{L}v$ term by term under the same hypothesis as in Theorem 1.1, we obtain
\begin{equation}
    \tag{3.49}
    0\ge  \mathcal{L}v\ge C\mathcal{F}|Du_1|^2-C(1+\mathcal{F})|Du_1|-C(1+\mathcal{F})(1+M_2),
\end{equation}
at the point $x_0\in\Omega$. Which implies 
\begin{equation}
    \tag{3.50}
    |Du_1|(x_0)\le C\sqrt{1+M_2},
\end{equation}
where $C$ is a constant depending on $A$ and, in addition, on $K|\phi|_{C^2;\Omega}$.

Finally, for any $\epsilon>0$, fixing $\phi$ so that $-K\phi\le\frac{\epsilon}{4}$ in $\Omega$. It is straightforward to see that $f$ has positive lower bound $\delta>0$. Since $x_0$ is the maximum point of $v$ on $\overline{\Omega}$, for any $x\in\overline{\Omega}$, it has
\begin{equation}
    \tag{3.51}
    v(x)\le v(x_0)\le \widetilde{v_1}(x_0)-f(x_0)\widetilde{v_1}(0)+\frac{\epsilon}{4}(1+M_2).
\end{equation}

Since $v_1(0)=\widetilde{v_1}(0)$, the assumption $v_1(0)\le0$ implies
\begin{equation}
    \notag
    u_{\tau\tau}\le u_{\tau\tau}+\frac{1}{2}|u_\tau|^2\le v_1(0)\le0.
\end{equation}
Thus every double tangential derivative on the boundary has upper bound 0, and the proof is complete. Now assume $v_1(0)\ge0$, Substituting $x=0$ into (3.51), we obtain
\begin{equation}
    \tag{3.52}
    \begin{split}
           0\le&\widetilde{v_1}(x_0)-f(x_0)\widetilde{v_1}(0)+\frac{\epsilon}{4}(1+M_2)\\
           \delta\widetilde{v_1}(0)\le&f(x_0)\widetilde{v_1}\le\widetilde{v_1}(x_0)+\frac{\epsilon}{4}(1+M_2)\\
           v_1(0)\le&\frac{1}{\delta}\widetilde{v_1}(x_0)+\frac{\epsilon}{4\delta}(1+M_2)\\
           v_1(0)\le&\frac{1}{\delta}(v_1(x_0)-g(0)\nu_1(x_0)-C_1(1+M_2)|x_0|^2)+\frac{\epsilon}{4\delta}(1+M_2)\\
           v_1(0)\le&C\sqrt{1+M_2}+C+\frac{\epsilon}{4\delta}(1+M_2)\\
            v_1(0)\le&\epsilon M_2+C_\epsilon,
    \end{split}
\end{equation}
where the constant $C_\epsilon$ depends on  $\epsilon,\gamma,A,B,G,\Omega,\beta_0$ and $||u||_{C^1;\Omega}$.  

Moreover, since $v_1(0)$ is the maximum of $v_\tau$ on the boundary under the constraint that $\tau$ is a unit tangent vector, we have
\begin{equation}
    \tag{3.53}
    \underset{\partial\Omega}{\sup}\underset{|\tau|\le1,\tau\cdot\nu=0}{\sup}u_{\tau\tau}\le\epsilon M_2+C_\epsilon.
\end{equation}

 \end{proof}
\subsection{Proof of Theorem 1.2}
\begin{proof}
For any unit vector $\xi\in R^n$, we set $s=\frac{\xi\cdot\nu}{\xi\cdot\beta}$, then we have $\tau=\xi-s\beta$, which is a tangential vector and
\begin{equation}
    \tag{3.54}
    u_{\xi\xi}=u_{\tau\tau}+2su_{\beta\tau}+s^2u_{\beta\beta}.
\end{equation}
 Form (1.2) and $|\beta|$ is bounded we know $s$ and $\tau$ is bounded. In the quasilinear case, by virtue of estimates (3.38), (3.17) and (3.18), under the hypothesis of Theorem 1.2, we then obtain the estimate
\begin{equation}
    \tag{3.55}
    \underset{\partial\Omega}{\sup}u_{\xi\xi}\le\epsilon M_2+C(\epsilon),
\end{equation}
for any constant $\epsilon>0$ and any unit vector $\xi\in R^n$, where $C(\epsilon)$ is a constant depending on $\epsilon,F,A,B,\gamma,G,\Omega,\beta_0$ and $||u||_{C^1;\Omega}$. The key hypothesis in this step is $\Gamma\subset\mathcal{P}_{n-1}$, which means the sum of any $n-1$ eigenvalues of $w=M[u]$ is nonnegative. Thus, this condition allows us to conclude that $M_2^+$ is qualitatively equivalent to $M_2'$, see \cite{jiang2018oblique} for more details.  

Then, given by admissibility (3.12), we can use the same method as (3.13)-(3.14) to get the lower bound of $D^2u$. Finally, we obtain
\begin{equation}
    \tag{3.56}
    \underset{\partial\Omega}{\sup}|D^2u|\le n\,\epsilon M_2+C(\epsilon),
\end{equation}
and the final estimate (1.12) follows by combining (3.56) with (1.10) or (3.15) and choosing $\epsilon$ sufficiently small. 
\end{proof}
 
\section{Gradient Estimates}

In this section, we prove the gradient estimate for admissible solutions $u$ of the oblique problem (1.1) under the assumption that the operator $G$ is semilinear and in particular give the proof of Theorem 1.3. Moreover we will employ improvements of the methods for uniformly elliptic equations in \cite{lieberman1986nonlinear}. The crucial point is how to use the tangential gradient terms. Before proving Theorem 1.3, we make some preparations first.

$G$ is semilinear, that means 
\begin{equation}
    \tag{4.1}
    G(x,z,p)=\beta(x)\cdot p-\varphi(x,z).
\end{equation}
Normalizing $G$: $\hat{G}:=\frac{G}{\beta\cdot\nu}=\hat{\beta}\cdot p-\hat{\varphi}$, and we relabel $\hat{G}$ as $G$ so that we can decompose the vector $\beta$ as $\beta=\beta'+\nu$, where $\beta'$ is a vector in the tangent plane of $\partial\Omega$ and $\nu$ is unit inner normal vector. Then we define the tangential gradient operator $\delta(\cdot)$ as follows.
\begin{equation}
    \tag{4.2}
     \delta(v)=Dv-(D_\nu v)\nu,\,   \mathrm{on} \  \partial\Omega. 
\end{equation}
For any function $v\in C^1(\overline{\Omega})$. The tangential gradient operator is a vector operator $\delta=(\delta_1,..,\delta_n)$ with $\delta_i=(\delta_{ij}-\nu_i\nu_j)D_j$. Under the assumptions of Theorem 1.3, $\beta'\in C^2(\partial\Omega),\varphi\in C^2(\partial\Omega\times\ \mathbb{R})$. Sicen  $\beta'$ is a vector in the tangent plane, we can rewrite (4.1) as 

\begin{equation}
    \tag{4.3}
     G(x,u,Du)=D_\nu u+\beta'\cdot\delta u-\varphi(x,z)=0, \,  \mathrm{on} \  \partial\Omega.
\end{equation}

For later use, we precompute several expressions. Given a function $g(x,z,p)$ as
\begin{equation}
    \tag{4.4}
    g(x,u,Du)=a_{kl}(x)u_ku_l+b_k(x,u)u_k+c(x,u),
\end{equation}
where $\{a_{kl}\}$ is a nonnegative matrix function on $\Omega$, $b_k(x,u),c(x,u)$ are scalar functions on $\Omega\times \mathbb{R}$, $a_{kl}\in C^2(\overline{\Omega}),b_k,c\in C^2(\overline{\Omega}\times \mathbb{R})$, and $u\in C^3(\Omega)$ is an admissible solution of (1.1). 

For the function $g$, directly computing its derivatives,  we have
\begin{equation}
    \tag{4.5}
     \begin{split}
    &g_{x_i}=D_{x_i}a_{kl}u_ku_l+D_{x_i}b_ku_k+D_{x_i}c,\ \ g_z=D_zb_ku_k+D_zc,\\
    &g_{p_k}=2a_{kl}u_l+b_k, \ \  \frac{\partial g}{\partial x_i}=g_{x_i}+g_zu_i+g_{p_k}u_{ki}.\\
    &g_{x_ix_j}=D_{x_jx_i}a_{kl}u_ku_l+D_{x_jx_i}b_ku_k+D_{x_jx_i}c,\ \ D_{p_l}D_{p_k}g=2a_{kl},\\
    &\Delta_xg=\Delta a_{kl}u_ku_l+\Delta_x b_ku_k+\Delta_x c,  \ \ D_{zz}g=D_{zz}b_ku_k+D_{zz}c,\\
    &\widetilde{D_{x_i}}D_{p_k}g=D_{x_i}a_{kl}u_l+\widetilde{D_{x_i}}b_k,\ \  \widetilde{D_{x_i}}D_{z}g=\widetilde{D_{x_i}}D_{z}u_k+\widetilde{D_{x_i}}D_{z}c,\\
    &\mathcal{L}g=Lg-D_{p_k}B(D_{x_k}g+D_zgu_k+D_{p_l}u_{kl}),
    \end{split}
\end{equation}
where $\widetilde{D}_{x_i}=D_{x_i}+u_iD_z$, $\Delta_xg(x,z,p)=\sum_{i=1}^n\frac{\partial^2 g}{(\partial_{x_i})^2}(x,z,p)$ and $\Delta_xb_k,\Delta_xc$ are the same meaning of $\Delta_xg(x,z,p)$.
Using (2.12) to compute the formula for $\mathcal{L}g$, after a straightforward computation, we obtain
\begin{equation}
    \tag{4.6}
    \begin{split}
      \mathcal{L}g&=(D_zb_ku_k+D_zc)\mathcal{L}u+(2a_{kl}u_l+b_k)\mathcal{L}u_k+u_ku_l\mathcal{L}a_{kl}\\
      &+2F^{ij}[\gamma\delta_{ij}a_{kl}\sum_mu_{km}u_{lm}-a_{kl}u_{ki}u_{lj}]+u_l[F^{ij}(\gamma\delta_{ij}\Delta b_l-D_{x_ix_j}b_l)-F^{ij}A_{ij}^kD_{x_k}b_l\\
      &-D_{P_k}BD_{x_k}b_l]+[F^{ij}(\gamma\delta_{ij}\Delta c-D_{x_ix_j}c-D_{x_k}c(F^{ij}A_{ij}^k+D_{p_k}B)]\\
      &+2\widetilde{F}^{ij}(2u_lD_{x_i}a_{kl}+\widetilde{D}_{x_i}b_k)u_{jk}
      +2\widetilde{F}^{ij}\gamma\delta_{ij}(\widetilde{D}_{x_i}D_zb_k+\widetilde{D}_{x_i}D_zc)u_j\\
      &-F^{ij}[\gamma\delta_{ij}(D_{zz}b_k+D_{zz}c)\sum_mu_m^2-(D_{zz}b_k+D_{zz}c)u_iu_j].
  \end{split}
\end{equation}
To estimate it, we need a new version of (2.11), namely
\begin{equation}
    \tag{4.7}
    \begin{split}
            \mathcal{L}u&=F^{ij}(\gamma\delta_{ij}\Delta u-\delta_{ik}u_{jk}-A_{ij}^{p_k}u_k)-D_{p_k}Bu_k,\\
            \mathcal{L}u_k&=F^{ij}\widetilde{D}_{x_k}A_{ij}+\widetilde{D}_{x_k}B.
            \end{split}
\end{equation}
Therefore, combining (4.6), (4.7), (2.14) and using the growth assumptions (1.13), it follows that
\begin{equation}
    \tag{4.8}
    \begin{split}
   \mathcal{L}g\ge &2F^{ij}[\gamma\delta_{ij}\sum_ma_{kl}u_{lm}u_{km}-a_{kl}u_{ki}u_{jl}]+\widetilde{F}^{ij}\widetilde{\beta}_{ik}u_{jk}
  \\&  -C(\mathcal{F}+1)(\omega(|Du|)|Du|^4+1),
    \end{split}
\end{equation}
where $\omega$ is a positive, monotonically decreasing function that tends to zero at infinity, depending on $A,B,\sup_\Omega|u|$ and $\widetilde{\beta}_{ik}=2(2u_lD_{x_i}a_{kl}+\widetilde{D}_{x_i}b_k)+(u_kD_zb_k+D_zc)\delta_{ik}$. If we assume that the matrix $\{a_{kl}\}$ is positive definite and satisfies $\{a_{kl}\}\ge a_1I$, then by applying Lemma 2.1, Cauchy's inequality and Young's inequality, with careful control of the coefficients, the estimate becomes
\begin{equation}
    \tag{4.9}
    \begin{split}
   \mathcal{L}g&\ge2a_1F^{ij}(\gamma\delta_{ij}\sum_mu_{km}u_{km}-u_{ki}u_{kj})-a_1\widetilde{F}^{ij}u_{ik}u_{jk}-C\widetilde{F}^{ij}\widetilde{\beta}_{ik}\widetilde{\beta}_{jk}+
   \\&-C(\mathcal{F}+1)(w(|Du|)|Du|^4+1)
   \\&=2a_1\widetilde{F}^{ij}u_{ki}u_{kj}-a_1\widetilde{F}^{ij}u_{ik}u_{jk}-C\widetilde{F}^{ij}\widetilde{\beta}_{ik}\widetilde{\beta}_{jk}-C(\mathcal{F}+1)(w(|Du|)|Du|^4+1)\\
    &\ge(\gamma-1)F^{ij}a_1\delta_{ij}\sum_mu_{km}^2-C\widetilde{F}^{ij}\widetilde{\beta}_{ik}\widetilde{\beta}_{jk}-C(\mathcal{F}+1)(w(|Du|)|Du|^4+1)\\
     &\ge a_1(\gamma-1)F^{ij}u_{ki}u_{kj}-C(\mathcal{F}+1)(w(|Du|)|Du|^4+1).
    \end{split}
\end{equation}
 The constant $C$ depends on $n,a_{kl},b_k,c,\Omega,A,B$ and $\sup_\Omega|u|$. Now we begin to prove Theorem 1.3.
\begin{proof}

For this problem, we introduce the auxiliary function
\begin{equation}
\tag{4.10}
    g:=|\delta u|^2+|G|^2=|Du|^2+|D_{\beta'}u|^2+2D_{\beta'}uD_\nu u-2D_\beta u\varphi+\varphi^2,
\end{equation}
where $\nu$ is the unit inner normal vector on $\partial\Omega$ and $G$ is given by (4.3). Then $\nu,G$ can be defined on $\overline{\Omega}$ by suitable extension. In essence, we extend the functions $\beta$, $\varphi$, and $\nu$ to the entire domain while preserving the form $G=D_\beta u-\varphi$.

Now the functions
 $g=a_{kl}(x)u_ku_l+b_k(x,u)u_k+c(x,u)$ have the explicit expression and $a_{kl}(x)=\delta_{kl}+\beta_k'\beta_l'+(\beta_k'\nu_l+\beta_l'\nu_k)$.
 The eigenvalues of the matrix $\{\beta_k'\beta_l'\}_{n\times n}$ are 
0 and $|\beta'|^2$. The eigenvalues of the matrix$\{(\beta_k'\nu_l+\beta_l'\nu_k)\}_{n\times n}$ are 
0, $|\beta'||\nu|$ and $-|\beta'||\nu|$. A direct analysis shows that the matrix $a_{kl}$
  is positive definite, with eigenvalues $1,1+\frac{|\beta'|^2}{2}\pm\frac{\sqrt{|\beta'|^4+4|\beta'|^2}}{2}$. Then it follows that
\begin{equation}
    \tag{4.11}
    \{a_{kl}\}\ge\frac{I}{1+|\beta'|^2+|\beta'|}\ge\frac{I}{(1+|\beta'|)^2}.
\end{equation}

We set $\beta_0'=\sup_\Omega|\beta'(x)|$, then (4.9) holds with $a_1=\frac{1}{(1+\beta_0')^2}$. Furthermore, applying Cauchy's inequality and Young's inequality yields
 \begin{equation}
     \tag{4.12}
    \frac{a_1}{2}|Du|^2-C_1\le\frac{a_1}{2}|Du|^2+(1-\frac{4}{a_1^2})|\varphi|^2\le g\le\frac{2}{a_1}|Du|^2+2|\varphi|^2\le\frac{2}{a_1}|Du|^2+C_2,
 \end{equation}
where the constant $C_1,C_2$ depend on $||\beta||_{C_0}$ and $||\varphi||_{C_0}$. Next, we introduce the final auxiliary function,
 \begin{equation}
     \tag{4.13}
     v:=g+M_1^2(\alpha\eta-\mathit{k}\phi),
 \end{equation}
where $\alpha,\mathit{k}$ are positive constants to be determined later and $\phi\in C^2(\Omega)$ is a negative defining function satisfying $\phi\equiv0$, $D_\nu\phi\equiv-1$ on $\partial\Omega$. $M_1=\sup_\Omega|Du|$ and $\eta$ is defined  as
 \begin{equation}
     \notag
     \eta=e^{K(u-u_0)},
 \end{equation}
 where $u_0:=\underset{\Omega}{\inf}\,u$ and $K$ is a positive constant to be determined later.

Due to the oblique boundary condition $<\beta,\nu>\ge\beta_0$ on $ \partial\Omega$, we can deduce that there exists a number $a\in(0,1)$ such that $D_\beta \phi<-a,\mathrm{on}\ \partial\Omega$.

Next, we analyze the maximum point $x_0$ of the function $v$ on $\overline{\Omega}$. We assume that $|Du|(x_0)>1$, otherwise, we can prove Theorem 1.3 immediately. Two cases need to be considered.

Case (i)：$x_0\in\partial\Omega$. So we have
\begin{equation}
    \tag{4.14}
    D_\beta v=D_\beta g +M_1^2(\alpha K\eta D_\beta u-\mathit{k}D_\beta \phi)\le0.
\end{equation}
 From the boundary condition (4.1) and $D_\beta \phi<-a$, we obtain
 \begin{equation}
     \tag{4.15}
     D_\beta g +M_1^2(\alpha K\eta D_\beta u-\mathit{k}D_\beta \phi)\ge D_\beta g +a\mathit{k}M_1^2-\alpha K\eta|\varphi|M_1^2.
 \end{equation}

Now we compute the expression of $D_\beta g$. By our definition, on $\partial\Omega$,
\begin{equation}
    \tag{4.16}
g=|Du-(D_\nu u)\nu|^2, \ \ Dg=D(|Du-(D_\nu u)\nu|^2).
\end{equation}

It is straightforward to verify that the vector $\tau_k=(-\nu_k\nu_1,-\nu_k\nu_2,....1-\nu_k^2,....,-\nu_k\nu_n)$ is a tangential vector on $\partial\Omega$. Differentiating the boundary condition (4.1) along the vector $\tau_k$ also yields 0 at $x_0$. This result implies 
\begin{equation}
    \tag{4.17}
   \begin{split}
        \delta_k\varphi&=\delta_{x_k}\varphi+(D_z\varphi)\delta_{k}u:=\widetilde{D}_{\tau_k}\varphi=D_{\tau_k}(D_\beta u)=(\delta_k\beta_i)u_i+\beta_iu_{ik}-\beta_i\nu_k\sum_jv_ju_{ij}\\&=
       (\delta_k\beta_i)u_i+\beta_iu_{ik}-\beta_i\nu_k\sum_jv_ju_{ij}-\beta_iD_i\nu_{k}D_\nu u-\beta_i\nu_kD_i\nu_ju_j+\beta_iD_i\nu_{k}D_\nu u+\beta_i\nu_kD_i\nu_ju_j\\
       &=(\delta_k\beta_i)u_i+D_\beta(\delta_ku)+\beta_iD_i\nu_{k}D_\nu u+\beta_i\nu_kD_i\nu_ju_j.
   \end{split}
\end{equation}
and it is not hard to prove that $|\delta u|^2=|Du|^2-\sum_{k,m}u_ku_m\nu_k\nu_m\le|Du|^2$. Using this fact, (4.17) and Cauchy's inequality, we obtain
 \begin{equation}
     \tag{4.18}
     \begin{split}
        D_\beta g&=2(\delta_ku)D_\beta\delta_ku\\
        &=2(\delta_ku)[\delta_k\varphi-(\delta_k\beta_i)u_i-\beta_iD_i\nu_{k}D_\nu u-\beta_i\nu_kD_i\nu_ju_j]\\
        &\ge-C(1+|Du|^2).
     \end{split}
 \end{equation}

Substituting (4.18) into (4.15), it becomes
\begin{equation}
    \tag{4.19}
    D_\beta v\ge-C(1+|Du|^2)+M_1^2[a\mathit{k}-\alpha K\eta|\varphi|(x,u)]\ge\frac{1}{2}\mathit{k}M_1^2-C,
\end{equation}
where we require $\mathit{k}\ge C(1+\alpha K\underset{\partial\Omega}{\sup}\,\eta)$ and $C$ is a constant depending on $\beta,\varphi,\Omega$ and $M_0:=\underset{\Omega}{\sup}|u|$.

Finally, from (4.14) and (4.19), we obtain
\begin{equation}
    \tag{4.20}
    M_1^2\le\frac{2C}{\mathit{k}}.
\end{equation}
Then the Gradient of admissible function $u$ is bounded.
  \vspace{0.4cm}

 Case（ii）$x_0\in\Omega.$ We change the requirements of constants $\alpha,\mathit{k}$ to 
\begin{equation}
    \tag{4.21}
    \alpha\le\frac{a_1}{16}\underset{\Omega}{\sup}\,\eta, \ \ \phi\ge-\frac{a_1}{16}\mathit{k}, \,  \mathrm{in} \  \Omega. 
\end{equation}
The purpose of this assumption is clear: if we can prove that $|Du|(x_0)\le C$ for some constant $C$, then it follows from the definition of $v$ that
\begin{equation}
    \tag{4.22}
    v(x_0)\le C'+\frac{a_1}{8}M_1^2,
\end{equation}
where constant $C'$ depends on $\beta,\varphi$ and C.  

Then because $x_0$ is the maximum point of $v$, it can show
\begin{equation}
\tag{4.23}
    g(x)\le v(x)=g(x)+M_1^2(\alpha\eta-\mathit{k}\phi)(x)\le C'+\frac{a_1}{8}M_1^2, \,  \forall x\in\overline{\Omega}.
\end{equation}
Substituting (4.23) into (4.12), we obtain
\begin{equation}
    \tag{4.24}
    \frac{a_1}{2}|Du|^2-C_1\le C'+\frac{a_1}{8}M_1^2, \,  \forall x\in\overline{\Omega}.
\end{equation}
The desired result then follows by taking the supremum on both sides over $\overline{\Omega}$.
\begin{equation}
    \tag{4.25}
    \frac{3a_1}{8}M_1^2\le C',
\end{equation}
where $C'$ is a constant depending on $\beta$, $\varphi$ and C. We now proceed to prove its existence. It suffices to prove $|Du|(x_0)\le C$. 

Since $x_0$ is the maximum point of $v$ in $\Omega$, $Dv(x_0)=0$ and
\begin{equation}
    \tag{4.26}
    \mathcal{L}v(x_0)=F^{ij}(\gamma\delta_{ij}\Delta v-v_{ij})(x_0)=\widetilde{F}^{ij}v_{ij}(x_0)\le0,
\end{equation}
where $\widetilde{F}^{ij}$ is defined as $\widetilde{F}^{ij}=\delta_{ij}\gamma\sum_{k}^nF^{kk}-F^{ij}=\gamma\delta_{ij}\mathcal{F}-F^{ij}$, the same as (1.8).

To analysis $\mathcal{L}v(x_0)$, the key point is to compute $\mathcal{L}\eta(x_0)$. We need to impose some new growth conditions as in Theorem 1.3, namely  
\begin{equation}
\tag{4.27}
    p\cdot A_p\le O(|p|^2)I, \ \ p\cdot B_p\le O(|p|^2),\ \ A=O(|p|^2),
\end{equation}
as $|p|\to\infty$, uniformly for $x\in\Omega,|z|\le M$ for any $M>0$. Then it follows that
\begin{equation}
    \tag{4.28}
    \begin{split}
        \mathcal{L}\eta&=F^{ij}(\gamma\delta_{ij}\Delta\eta-\eta_{ij}-A_{ij}^{p_k}D_k\eta)-D_{p_k}BD_k\eta  \\
        &=F^{ij}[\gamma\delta_{ij}(K\eta\Delta u+K^2\eta|Du|^2)-K^2\eta u_iu_j-K\eta u_{ij}-K\eta A_{ij}^{p_k}u_k]-D_{p_k}BK\eta u_k\\
         &=K\eta[\widetilde{F}^{ij}u_{ij}+K\widetilde{F}^{ij}u_iu_j-(F^{ij}A_{ij}^{p_k}+D_{p_k}B)u_k]\\
         &\ge K\eta[\widetilde{F}^{ij}u_{ij}+K\widetilde{F}^{ij}u_iu_j-\mu(\mathcal{F}+1)(1+|Du|^2)],
    \end{split}
\end{equation}
where constant $\mu$ depends on $M_0$. From Sect. 1, we know that the eigenvalues of the matrix $\{\widetilde{F}^{ij}\}$ are $\gamma\mathcal{F}-\lambda_i$, where $\lambda_i$ are the eigenvalues of the matrix $\{F^{ij}\}$. Consequently, the smallest eigenvalue of this matrix is greater than $(\gamma-1)\mathcal{F}$. That means
\begin{equation}
    \tag{4.29}
    \widetilde{F}^{ij}u_iu_j\ge(\gamma-1)\mathcal{F}|Du|^2.
\end{equation}

According to assumption (1.6) $F_{ij}w_{ij}=F=B\ge0$, Using (4.27) we can obtain 
\begin{equation}
    \tag{4.30}
    \widetilde{F}^{ij}u_{ij}\ge F^{ij}A_{ij}\ge-C\mu_1\mathcal{F}|Du|^2,
\end{equation}
where $\mu_1$ is a constant depending on $M_0=\underset{\Omega}{\sup}|u|$.

According to assumption $F4$, $\mathcal{F}>\delta_0$, we can fix the constant $K$ such that $K>\frac{4\mu(1+\frac{1}{\delta_0})+2C\mu_1}{\gamma-1}$. 

Plugging (4.30) and (4.29) into (4.28), we can obtain

\begin{equation}
    \tag{4.31}
    \mathcal{L}\eta\ge \frac{1}{2}K\eta(1+\mathcal{F})|Du|^2.
\end{equation}

We now proceed to rewrite (4.26) in full detail. Using (4.9), (4.31) and growth conditions, we obtain
\begin{equation}
    \tag{4.32}
    \begin{split}
    0\ge\mathcal{L}v(x_0)&=\mathcal{L}g+\alpha M_1^2\mathcal{L}\eta-\mathit{k}M_1^2\mathcal{L}\phi\\
    &\ge\frac{a_1(\gamma-1)}{2}F^{ij}u_{ki}u_{kj}-C(\mathcal{F}+1)(w(|Du|)|Du|^4+1)\\
    &+\alpha M_1^2\frac{K}{2}(1+\mathcal{F})|Du|^2-C\mathit{k}M_1^2(1+\mathcal{F})w(|Du|)|Du|^2\\
     &\ge(1+\mathcal{F})[\alpha M_1^2\frac{K}{2}|Du|^2-C(w(|Du|)|Du|^4+1+\mathit{k}M_1^2w(|Du|)|Du|^2)].
    \end{split}
\end{equation}
For sufficiently large $|Du|$, the quantity $w(|Du|)$ becomes arbitrarily small. As the parameters $\alpha,\mathit{k}$ and $K$ are fixed, we can assume $|Du|(x_0)$ is sufficiently large such that $C(1+k)\omega(|Du|)\le\frac{\alpha K}{4}$, otherwise $|Du|(x_0)$ is bounded, we already get the result.  Consequently, we obtain
\begin{equation}
    \tag{4.33}
    \frac{\alpha M_1^2K}{4}|Du|^2(x_0)\le C\to|Du|(x_0)\le C,
\end{equation}
where the constant $C$ depends on $F,A,B,\Omega,\gamma,b_0,G$ and $M_0$.
This completes the proof of Case (ii) and consequently establishes the boundedness of $M_1=\underset{\Omega}{\sup}|Du|$.
    
\end{proof}

\section{Proof of Theorem 1.4}
\begin{proof}
After we complete the proof of Theorem 1.2 and Theorem 1.3, it remains to establish $C^0$ estimates. Under the hypothesis of Theorem 1.4, $C^0$ estimates can be directly obtained by the existence of the admissible subsolution $\overline{u}$ and supersolution $\underline{u}$. Moreover, if $u\in C^2(\Omega)\cap C^1(\overline{\Omega})$ is an admissible solution of the equation (1.1), by comparison principle, we have
\begin{equation}
\tag{5.1}
\underline{u}\le u\le\overline{u},\, \mathrm{in}\, \ \overline{\Omega}.
\end{equation}
 Then based on the a priori estimates for problems (1.1) in this paper, $C^{2,\alpha}$  $(0<\alpha<1)$ estimates can be achieved by Evans-Krylov theory and higher order estimates followed by Schauder thoery. The existence result can be concluded by the continuity method and using the monotonicity conditions on $A,B,\varphi$ we can derive the uniqueness by classical comparison principle. The higher regularity of functions $\beta$ and $\varphi$ is needed for the proof. More details see \cite{gilbarg1977elliptic}. Theorem 1.4 can be proven by approximation as in \cite{jiang2018oblique}.

\end{proof}

\bibliographystyle{plainnat}
\bibliography{sample}
\end{document}